\documentclass[12pt]{amsart}

\usepackage{amsmath,amsfonts,amsthm,amssymb}
\usepackage{hyperref}
\usepackage{tikz}
\usetikzlibrary{arrows.meta, positioning, calc}
\usepackage{pgfplots}
\usepackage{pgfplotstable} 
\usepgfplotslibrary{fillbetween}
\pgfplotsset{compat=1.18}
\usepackage{comment}
\usepackage{xcolor}
\usepackage{dirtytalk}
\usepackage{xfrac}

\usepackage{a4wide}
\usepackage{enumerate}
\usepackage{multicol, multirow, longtable}
\newtheorem{theorem}{Theorem}[section]
\newtheorem{lemma}[theorem]{Lemma}
\newtheorem{proposition}[theorem]{Proposition}
\newtheorem{corollary}[theorem]{Corollary}

\newtheorem{theo}{Theorem}

\theoremstyle{remark}

\newtheorem{remark}[theorem]{Remark}

\newtheorem*{claim*}{Claim}

\theoremstyle{definition}

\numberwithin{equation}{section}

\newcommand{\C}{\ensuremath{\mathbb{C}}}
\newcommand{\R}{\ensuremath{\mathbb{R}}}

\DeclareMathOperator{\Ad}{Ad}

\DeclareMathOperator{\Ric}{Ric}
\DeclareMathOperator{\scal}{scal}

\renewcommand{\H}{\ensuremath{\mathbb{H}}}

\newcommand{\spin}[1]{\ensuremath{\mathsf{Spin}_{#1}}}
\newcommand{\su}[1]{\ensuremath{\mathsf{SU}_{#1}}}
\renewcommand{\u}[1]{\ensuremath{\mathsf{U}_{#1}}}

\renewcommand{\sp}[1]{\ensuremath{\mathsf{Sp}_{#1}}}

\newcommand{\res}{\operatorname{res}}

\newcommand{\lieG}{{\mathfrak{g}}}
\newcommand{\lieH}{{\mathfrak{h}}}

\newcommand{\lieP}{{\mathfrak{p}}}

\usepackage{parskip}

\makeatletter
\newsavebox{\@brx}
\newcommand{\llangle}[1][]{\savebox{\@brx}{\(\m@th{#1\langle}\)}%
	\mathopen{\copy\@brx\kern-0.5\wd\@brx\usebox{\@brx}}}
\newcommand{\rrangle}[1][]{\savebox{\@brx}{\(\m@th{#1\rangle}\)}%
	\mathclose{\copy\@brx\kern-0.5\wd\@brx\usebox{\@brx}}}
\makeatother

\begin{document}
\title{Ricci flow preserves positive sectional curvature on homogeneous spheres}

\author[J.~Devito]{Jason DeVito}
\address{Jason DeVito\\The University of Tennessee at Martin, USA}
\email{jdevito1@ut.utm.edu}
\author[D.~Gonz\'alez-\'Alvaro]{David Gonz\'alez-\'Alvaro}
\address{David Gonz\'alez-\'Alvaro\\Universidad Polit\'ecnica de Madrid, Spain.}
\email{david.gonzalez.alvaro@upm.es}
\author[M.~Zarei]{Masoumeh Zarei}
\address{Masoumeh Zarei\\Fachbereich Mathematik, Universität Hamburg, Hamburg, Germany}
\email{masoumeh.zarei@uni-hamburg.de}

\begin{abstract}
We prove that the Ricci flow preserves positive sectional curvature on homogeneous spheres and complex projective spaces. In conjunction with prior results, this completes the classification of which homogeneous spaces have positively curved metrics flowing outside the set of positively curved metrics and which do not.

\end{abstract}


\maketitle

\section{Introduction}

Ricci flow is an evolution equation for Riemannian metrics on smooth manifolds, introduced by Hamilton~\cite{Ham82} to deform an initial metric towards a canonical or ``optimal" one. Because canonical metrics are characterized by distinguished curvature properties, understanding the evolution of curvature under the flow has been a central theme in the theory.

Positive scalar curvature is known to be preserved under the Ricci flow in all dimensions, and the same is true for the much stronger condition of positive curvature operator. Between these two notions lie many intermediate curvature conditions whose preservation under the flow is still not fully understood in all dimensions. In particular, there exist examples where positive Ricci and sectional curvature are not preserved. More generally, there exist infinitely many dimensions where certain intermediate curvature conditions fail to hold along the flow \cite{AN16,GZ24}.

In this article we restrict to the classical notion of positive sectional curvature, denoted by $\sec>0$. In his foundational paper \cite{Ham82}, Hamilton proved that $\sec>0$ is preserved in dimensions $2$ and $3$. In contrast, later works indicate that one should not expect the preservation of $\sec>0$ in any dimension $\geq 4$, see e.g. \cite[Remark~5.1]{BeK23}. However, constructing examples remains very challenging, and has been achieved only in dimensions $4$, $6$, $7$, $12$, $13$ and $24$ \cite{BeK23,BW07,CW15,GZ25}. Indeed, one needs an example of a manifold of $\sec >0$ to begin with, which so far has been a very complicated task itself, with the additional difficulty of being able to control the evolution of curvature.

The known examples of simply connected manifolds with $\sec > 0$ are compact rank one symmetric spaces, namely the projective spaces $\mathbb{S}^{n},\C P^{n},\H P^n$ and the Cayley plane, and certain spaces in dimensions $6$, $7$, $12$, $13$, $24$. Among them, the existing examples of metrics where $\sec >0$ is not preserved under the flow have a large isometry group, which is crucial to control the evolution of curvature. In fact, with the exception of the examples in dimension $4$, all of them are homogeneous. Our main result shows that compact rank one symmetric spaces cannot produce such examples.

\begin{theo}\label{THM:main}
Let $g$ be a homogeneous metric of $\sec>0$ on a simply connected compact rank one symmetric space. Then the Ricci flow of $g$ has $\sec>0$ for all positive times. 
\end{theo}

Together with previous results, Theorem~\ref{THM:main} completes the classification of which manifolds admit a homogeneous $\sec>0$ metric evolving to metrics with non-positively curved planes under the flow.  To describe it, let us review the well-known classification of homogeneous spaces with $\sec>0$ (see e.g. \cite{WZ18}).

In the simply connected case, any such manifold is diffeomorphic to a compact rank one symmetric space, a Berger space $B^7,B^{13}$, a Wallach space $W^6,W^{12},W^{24}$, or an Aloff-Wallach space $W_{p,q}^7$. The space $B^7$ admits a unique homogeneous metric up to scaling \cite[p. 83]{Z-Survey}. This  metric is Einstein and hence is simply rescaled under the Ricci flow, so $\sec >0$ is trivially preserved.  In contrast, the rest of $B^{13}$, $W^6$, $W^{12}$, $W^{24}$, $W_{p,q}^7$ have been shown to carry $\sec>0$ metrics that flow to metrics with tangent planes of non-positive curvature.

In the non-simply connected case, the universal covering must be diffeomorphic to one of the spaces $\mathbb{S}^n$, $\mathbb{C}P^{n}$, $W^6$, $W^{12}$, $W^{24}$ or $W_{p,q}^7$, see Section~\ref{sec:thmcd} for details on the classification. The homogeneous Ricci flow is invariant under Riemannian coverings, thus we can use the existing results for the simply connected case, provided the metrics descend to the corresponding quotient. For quotients of $\mathbb{S}^n$ and $\mathbb{C}P^{n}$ it follows from Theorem~\ref{THM:main} that $\sec>0$ is preserved under the flow. For quotients of $W^6$, $W^{12}$, $W^{24}$ or $W_{p,q}^7$, an inspection of the metrics given in \cite{BW07,CW15,GZ25} shows that, for each   admissible quotient, there exists at least one $\sec >0$ metric on the universal covering space which descends to  the quotient,  and evolves to metrics with non-positively curved planes. Altogether we derive the following result.

\begin{theo}\label{THM:homogeneous_spaces}
A compact manifold $M$ has a homogeneous metric $g$ of $\sec>0$ for which its Ricci flow evolves to metrics with non-positively curved planes if and only if the universal cover of $M$ is diffeomorphic to a Wallach space $W^6$, $W^{12}$, $W^{24}$, an Aloff-Wallach space $W_{p,q}^7$, or the Berger space $B^{13}$.
\end{theo}

Theorem~\ref{THM:homogeneous_spaces} closes the search of examples where $\sec >0$ is not preserved in the homogeneous case and indicates the difficulty in finding new ones because the action of the isometry group would be of cohomogeneity at least one, where analyzing the evolution of curvature is much more complicated. So far this has only been achieved by Bettiol and Krishnan in \cite{BeK23} among cohomogeneity one metrics on $\mathbb S^4$ and $\mathbb C P^2$.

We now outline the proof of Theorem \ref{THM:main}. First, since Ricci flow preserves isometries, the flow of any homogeneous metric stays within the set of homogeneous metrics. Since $\mathbb HP^n$ and the Cayley plane admit a unique homogeneous metric up to scaling, we need only investigate the cases of $\mathbb{S}^n$ and $\mathbb{C}P^{n}$. Thus, we begin with a parameterization, due to Ziller \cite{Zi82}, of all possible homogeneous metrics on either $\mathbb{S}^n$ and $\mathbb{C}P^{n}$.  For $\mathbb{S}^n$, the subset $\mathcal{P}$ of those with positive sectional curvature was determined by Verdiani and Ziller \cite{VZ09} and is described by the positivity of a certain set of functions of the parameters.  We derive the corresponding set for $\mathbb{C}P^n$ in Proposition~\ref{prop:cppos}.  Now, we take a metric $g$ on the boundary of $\mathcal{P}$.  Because the metric is homogeneous, the Ricci flow PDE reduces to a system of ODEs, which, in particular, implies that the flow exists for short negative time.  By analyzing the sign of derivatives of the functions which define $\mathcal{P}$ we show that the backwards flow of $g$ deforms to a metric outside of $\mathcal{P}$. By uniqueness of flows, we deduce that the only metrics whose forward-time flows hit $\partial \mathcal{P}$ lie outside of $\mathcal{P}$.  Thus, we conclude that no flow beginning at a positively curved metric can ever reach $\partial \mathcal{P}$, completing the proof.

\textbf{Organization of the paper.} In Section~\ref{SEC:resultant_Sturm}, we recall the background on resultants and the Sturm sequence, two of the main tools we use to analyze the sign of derivatives.  In Section~\ref{SEC:HRF}, we recall the basics of the homogeneous Ricci flow and then precisely state and prove our main tool for showing that the set $\mathcal{P}$ of $\sec>0$ metrics is preserved.  In Section~\ref{SEC:homog_metrics}, we describe the set of homogeneous metrics on $\mathbb{S}^n$ and $\mathbb{C}P^n$ and explicitly determine the Ricci flow ODE of a general homogeneous metric, see \eqref{EQ:RF_2param} and \eqref{EQ:RF_4param}.  Section~\ref{sec:2param} contains the proof that $\mathcal{P}$ is preserved in the case where $\mathcal{P}$ is determined by a single function, see Theorem~\ref{THM:2-param}.  This, includes the case of $\mathbb{C}P^n$.

Sections~\ref{SEC:4param} through \ref{SEC:derivative_of_P1} are devoted to the proof that $\mathcal{P}$ is preserved in the case where $\mathcal{P}$ is defined by nine functions, see Theorem~\ref{THM:4-param}.  Beginning in Section~\ref{SEC:4param}, we precisely define the set $\mathcal{P}$ and introduce a new set of parameters which simplifies later calculations.  Crucially, we reduce the process of checking all nine functions to analyzing only three.  In Sections~\ref{sec:H} and \ref{sec:V}, we carry out the analysis of the first two of these three functions.  The analysis of the third is considerably more difficult and is split over Sections~\ref{SEC:P1_0} and \ref{SEC:derivative_of_P1}.  In the first of these sections, we explicitly determine a relatively small subset containing the zero set of this function.  In the second, we analyze the derivative when restricted to this small subset, finally proving Theorem~\ref{THM:4-param}.

We conclude the article with a proof of Theorem~\ref{THM:homogeneous_spaces} in Section~\ref{sec:thmcd}.

\textbf{Acknowledgments.} This work was initiated during the Workshop on Curvature and Global Shape, held in M\"unster in 2025, for which we thank the organizers for excellent working conditions. The authors would like to thank Renato Bettiol, Christoph B\"ohm and Wolfgang Ziller for helpful conversations. In addition, the first author gratefully acknowledges support from NSF DMS-2405266. The second author was supported by grant PID2024-158664NB-C21 from the Agencia Estatal de Investigación and the Ministerio de Ciencia, Innovación y Universidades
(Spain). The third author acknowledges support from the Deutsche Forschungsgemeinschaft (DFG, German Research Foundation) under Germany’s Excellence Strategy -- EXC 2121 ``Quantum Universe'' -- 390833306.

\section{Resultants and the Sturm sequence}\label{SEC:resultant_Sturm}

In this section, we outline two classical techniques in the study of the zeros of polynomial systems of equations: resultants and Sturm sequences. Note that most of our computations involving resultants and Sturm sequences are carried out using Maple. We therefore also indicate the specific Maple commands that are used.

We begin with resultants. Suppose $R$ is an integral domain and consider two polynomials $f(x),g(x)\in R[x]$.  Then we may write $f(x) = \sum_{i=0}^n a_i x^i$ and $g(x) = \sum_{j=0}^m b_j x^j$ where $a_n$ and $b_m$ are both non-zero.  Then the \textit{resultant} of $f$ and $g$, denoted by $\res(f,g)$, is defined as the determinant of the corresponding $(m+n)\times (m+n)$ Sylvester matrix: 
$$
\operatorname{res}(f,g) = \det \begin{bmatrix} a_n & a_{n-1} & a_{n-2} &  \cdots & a_0 & 0 & 0 & \cdots & 0 & 0\\
0 & a_n & a_{n-1} & \cdots & a_1 & a_0 & 0 & \cdots & 0 & 0 \\ 
0 & 0 & a_n & \cdots & a_2 & a_1 & a_0 &\cdots & 0 & 0\\
\vdots & \vdots & \vdots & \vdots & \vdots &\vdots & \vdots & \vdots & \vdots & \vdots\\
0 & 0 & 0 & \cdots  &\cdots & \cdots&\cdots & \cdots &  a_1 & a_0\\
 b_m & b_{m-1} & b_{m-2} &  \cdots & b_0 & 0 & 0 & \cdots & 0 & 0\\
0 & b_m & b_{m-1} & \cdots & b_1 & b_0 & 0 & \cdots & 0 & 0 \\ 
0 & 0 & b_m & \cdots & b_2 & b_1 & b_0 &\cdots & 0 & 0\\
\vdots & \vdots & \vdots & \vdots & \vdots &\vdots & \vdots & \vdots & \vdots & \vdots\\
0 & 0 & 0 & \cdots  &\cdots & \cdots&\cdots & \cdots & b_1 & b_0\\ 
\end{bmatrix}\in R,
$$ 
where the first $m$ rows are cyclic permutations of $(a_n,a_{n-1}, a_{n-2},..., a_0,0,...0)$ and the last $n$ rows are cyclic permutations of $(b_m, b_{m-1},...,b_0,0,...,0)$.  If one works in an algebraically closed field containing $R$, the resultant of $f$ and $g$ is equal to $a_n^m b_m^n \prod_{i=1}^n\prod_{j=1}^m(\alpha_i-\beta_j)$, where the $\alpha_i$'s are the roots of $f$ and the $\beta_j$'s are the roots of $g$. Clearly, then $\operatorname{res}(f,g) = 0$ if and only if $f$ and $g$ share a common root in some algebraically closed field containing $R$.

We are specifically interested in understanding the common zeros of two polynomials $f,g\in \mathbb{Z}[a,b] \cong (\mathbb{Z}[a])[b] \cong (\mathbb{Z}[b])[a]$.  Thus, we take $R = \mathbb{Z}[a]$ or $R = \mathbb{Z}[b]$.  We use the notation $\operatorname{res}_b(f,g)\in \mathbb{Z}[a]$ when viewing $f$ and $g$ as polynomials in the variable $b$ with coefficients in $\mathbb{Z}[a]$.  Then $\operatorname{res}_b(f,g)$ is single-variable polynomial in $a$, and if $(a_0,b_0)$ is a common zero of $f$ and $g$, then $a_0$ must be a zero of $\operatorname{res}_b(f,g)$. In Maple, given two polynomials $f,g$ in the variables $a$ and $b$, the command \verb+resultant(f,g,b)+ computes $\operatorname{res}_b(f,g)$.

Similarly, we use the notation $\operatorname{res}_a(f,g)\in \mathbb{Z}[b]$ to denote the resultant of $f$ and $g$ when viewed as polynomials in $a$ with coefficients in $\mathbb{Z}[b]$.  It is a single-variable polynomial in $b$ with the property that if $(a_0,b_0)$ is a common zero of $f$ and $g$, then $b_0$ is a zero of $\operatorname{res}_a(f,g)$.

\bigskip

We now turn attention to Sturm sequences.  To this end,  suppose that $f\in \mathbb{R}[x]$ is a single-variable real polynomial. From $f$, we construct a finite length sequence of polynomials $f_1,f_2,...$ of decreasing degree, called the \textit{Sturm sequence} of $f$, as follows.  The first term is $f_1 = f$ and the second term is the derivative, $f_2 = f'$.  Assume $f_i$ and $f_{i+1}$ are defined, we define $f_{i+2}$ to be $-\operatorname{rem}(f_i,f_{i+1})$, where $\operatorname{rem}$ denotes the remainder upon dividing $f_i$ by the smaller degree $f_{i+1}$.  We stop the process once we reach some $f_i$ of degree $0$.

Given an element $b\in \mathbb{R}$, the Sturm sequence of $f$ yields a real-valued finite-length sequence $f_1(b),f_2(b),\dots,f_k(b)$.  Deleting any occurrences of $0$ from this sequence, we let $V(b)$ denote the number of times the sign changes in the resulting sequence. In order to state Sturm's Theorem, recall that a polynomial $f$ is called square-free if it cannot be factored as $g^2 h$ for a non-unit (i.e., a positive degree) polynomial $g$ and $h$. Observe that every root of a square-free polynomial has multiplicity one.

\begin{theorem}[Sturm's Theorem]\label{T:Sturm_root}
If $f$ is square-free and $a<b$ then the number $V(b) - V(a)$ is equal to the number of roots of $f$ in the half open interval $(a,b]$.
\end{theorem}

\begin{remark}\label{REM:Sturm}
In practice, instead of checking the hypothesis that $f$ is square-free, one applies Sturm's Theorem not to $f$, but rather to $f/\operatorname{gcd}(f,f')$.  This polynomial is square-free and it has the same roots (with possibly different multiplicities) as $f$. In Maple, if $f$ is a single variable polynomial in the variable $x$, then \verb+sturmseq(f,x)+ computes the Sturm sequence of $f/\gcd(f,f')$. Calling the output of this $s$, then \verb+sturm(s,x,a,b)+ computes $V(b)-V(a)$.
\end{remark}

\section{Homogeneous Ricci flow and invariance of positive curvature}\label{SEC:HRF}

In this section we first recall some basic properties of the homogeneous Ricci flow (see e.g. \cite{La13} for details). We then study the conditions under which it preserves positive curvature.

Let $(M,g)$ be a compact homogeneous Riemannian manifold, and let $(M,g(\ell))$ be the Ricci flow of $g$, i.e. $g(0)=g$. Note that we are using $\ell$ to denote the time parameter in the Ricci flow, since the more natural variable $t$ will be used later to parameterize the metric. Because $(M,g)$ is homogeneous, there is a compact Lie group $G$ acting on $(M,g)$ transitively and by isometries. Since the isometry group is preserved under the flow, there is a presentation of $M$ as a homogeneous space $(M, g(\ell)) = (G/H, g (\ell))$ for all $\ell$ with the same reductive decomposition $\lieG= \lieH\oplus \lieP$, where $\lieH$ and $\lieG$ denote the Lie algebras of $H$ and $G$, and $\lieP$ denotes the orthogonal subspace to $\lieH$ with respect to some bi-invariant metric on $G$. Recall that the tangent space of $G/H$ at the origin is identified with $\lieP$.

Let $\langle\cdot,\cdot\rangle_{\ell} =g(\ell)(o)$ denote the  $\Ad(H)$-invariant inner product on $\lieP$, which is determined by the $G$-invariant metric $g(\ell)$ at the origin $o\in G/H$. Then the Ricci flow equation is equivalent to the ODE
\begin{equation}\label{Ric_ODE}
\frac{d}{d \ell}\langle\cdot,\cdot\rangle_{\ell}  = -2\Ric (\langle\cdot,\cdot\rangle_{\ell}), 
\end{equation}
where $\Ric(\langle\cdot,\cdot\rangle_{\ell}) = \Ric(g(\ell))(o)$. It follows that the maximal interval of existence of $g(\ell)$ is of the form $ (L_{-}, L_{+})$, where $-\infty\leq L_{-}<0<L_{+} \leq \infty$.

Suppose that $M$ admits homogeneous metrics of positive sectional curvature. The following result gives sufficient conditions for the homogeneous Ricci flow to preserve positive curvature.

\begin{proposition}\label{prop:RF_invariance}
    Let $M$ be a compact manifold with a transitive action of  a compact Lie group $G$, and let $\mathcal R$ be the space of $G$-invariant metrics. Assume that the subspace $\mathcal P\subset\mathcal R$ of $\sec>0$ metrics is non-empty, and let $\partial\mathcal P$ be  its boundary in $\mathcal R$. Suppose that for each $g\in\partial \mathcal P$, there is some $\varepsilon>0$  such that $g(\ell)\notin\mathcal P$ for all $\ell\in (-\varepsilon,0)$, where $g(\ell)$ denotes the Ricci flow of $g=g(0)$. Then the Ricci flow of every metric in $\mathcal P$ stays in $\mathcal P$ for all positive times of existence. 
\end{proposition}

\begin{proof}
     Suppose $h\in\mathcal P$ and consider its Ricci flow $h(\ell)$ with $h(0) = h$ and $\ell\in (L_-,L_+)$, where $L_-<0$ and $L_+>0$. Assume for a contradiction that $h(\ell_0) \notin\mathcal P$ for some $\ell_0\in (0,L_+)$.   Because $\mathcal{P}$ is open, there is minimum such time; without loss of generality we assume $\ell_0$ denotes this time.  It follows that $h(\ell_0)\in \partial\mathcal{P}.$ 

     Let $g:=h(\ell_0)$ and let $g(\ell)$ denote its Ricci flow.  Then by uniqueness of the Ricci flow, $g(\ell) = h(\ell+\ell_0)$ for all $\ell\in(L_--\ell_0,L_+-\ell_0)$.  By assumption, as $g(0)= h(\ell_0)\in \partial \mathcal{P}$,  there is some $\varepsilon>0$ such that $g(\ell)\notin\mathcal P$ for all $\ell\in (-\varepsilon,0)$. Then, for any $\ell\in(-\varepsilon,0)$ with $|\ell|<\ell_0$, we have $0< \ell+\ell_0 < \ell_0$ and $h(\ell + \ell_0) \notin\mathcal P$. This implies that the Ricci flow of $h$ leaves $\mathcal{P}$ at a time earlier than $\ell_0$, contradicting the choice of $\ell_0$.
\end{proof}

\section{Homogeneous metrics on CROSSes and their Ricci flow}\label{SEC:homog_metrics}

Homogeneous metrics on compact rank one symmetric spaces (CROSSes) have been classified. The spaces $\mathbb{S}^{2n}$, $\C P^{2n}$, $\H P^n$, and the Cayley plane only admit one homogeneous metric, up to scaling. This metric is symmetric and hence Einstein, so the Ricci flow obviously preserves $\sec>0$.

Thus, the only non-trivial cases for this work are odd-dimensional spheres $\mathbb{S}^{2n+1}$ and complex projective spaces $\C P^{2n+1}$, where $n\geq 1$. There are two ways to describe homogeneous metrics on them: via Hopf fibrations, or using their homogeneous space structure. Here we use the Hopf fibration description, and we refer to \cite[Section~2]{BPRZ21} for further details.

For the odd-dimensional spheres we use Hopf fibrations:
\begin{align*}
\mathbb S^{1}\to\mathbb S^{2n+1}\to\mathbb C P^n,\qquad \mathbb S^{3}\to\mathbb S^{4n+3}\to\mathbb H P^n,\qquad \mathbb S^{7}\to\mathbb S^{15}\to\mathbb S^{8}.
\end{align*}
Recall that these fibrations become Riemannian submersions when endowing the total space with the round metric $g_{\text{round}}$ of constant curvature one. Now, using the submersion structure, in each case we can define a metric $g_{t,s}$ by scaling $g_{\text{round}}$ by $t$ in the vertical directions and by $s$ in the horizontal directions, where $(t,s)\in\R_+^2$. Throughout this work $\mathbb{R}^k_+$ denotes the set of points in $\mathbb{R}^k$ whose coordinates are all positive. 

Apart from this two-parameter families of metrics, there is a four parameter family of metrics $g_{x,y,z,s}$ on $\mathbb{S}^{4n+3}$ associated to $\mathbb S^{3}\to\mathbb S^{4n+3}\to\mathbb H P^n$, with $(x,y,z,s)\in\R_+^4$. As we will explain, the homogeneous structure on $\mathbb {S}^{4n+3}$ yields a seven-parameter family of metrics; however, up to isometry, these metrics can be taken to be part of a four parameter family $g_{x,y,z,s}$ which is preserved under the Ricci flow. Here $s$ denotes the scaling of $g_{\text{round}}$ in the horizontal directions and  $(x,y,z)$ corresponds to scalings in vertical directions.

We now turn to the case of complex projective spaces $\mathbb{C}P^{2n+1}$, which admit a two-parameter family of homogeneous metrics.  In order to describe them, we begin with the Hopf fibration $\mathbb{S}^3\rightarrow \mathbb{S}^{4n+3}\rightarrow \mathbb{H}P^n$, a principal $\mathbb{S}^3$-bundle.  If we equip $\mathbb{S}^{4n+3}$ with the above metric $g_{t,s}$, the $\mathbb{S}^3$-action is isometric and free and thus so is the restriction of this action to  $\mathbb{S}^1\subseteq \mathbb{S}^3$.  Then the quotient $\mathbb{C}P^{2n+1} = \mathbb{S}^{4n+3}/\mathbb{S}^1$ inherits a metric $g^{\mathbb{C}P}_{t,s}$, which we will abbreviate as $g_{t,s}$ when there is no risk of confusion.  The canonical $\sp{n}$-action on $\mathbb{S}^{4n+3}$ commutes with the $\mathbb{S}^3$-action, so it descends to a transitive isometric action on $(\mathbb{C}P^{2n+1},g^{\mathbb{C}P}_{t,s})$, so these metrics are all homogeneous.

Alternatively, by quotienting the bundle 
$$
\mathbb{S}^3\rightarrow \mathbb{S}^{4n+3}\rightarrow \mathbb{H}P^n
$$ 
by the $\mathbb{S}^1$-subaction, we obtain a bundle 
$$
\mathbb{S}^2\rightarrow \mathbb{C}P^{2n+1}\rightarrow \mathbb{H}P^n.
$$
The metric $g^{\mathbb{C}P}_{t,s}$ is obtained by scaling the Fubini-Study metric on $\mathbb{C}P^{2n+1}$ by a factor of $t$ in the direction of the $\mathbb{S}^2$ fibers and by a factor of $s$ in the horizontal direction.

Next we discuss the Ricci tensor of the metrics above. We start with $g_{t,s}$, either on a sphere or on a complex projective space, which we denote generically by $M$. We need to recall the homogeneous structure of $M$. There is a transitive $G$-action on $M$ where $G$ equals $\su{n+1}$ in the case of  $\mathbb S^{1}\to\mathbb S^{2n+1}\to\mathbb C P^n$, $G=\sp{n+1}\sp{1}$ in the case of $\mathbb S^{3}\to\mathbb S^{4n+3}\to\mathbb H P^n$, $G=\spin{9}$ in the case of $\mathbb S^{7}\to\mathbb S^{15}\to\mathbb S^{8}$, and $G=\sp{n+1}$ in the case of $\mathbb{S}^2\rightarrow \mathbb{C}P^{2n+1}\rightarrow \mathbb{H}P^n$. The corresponding isotropy group is $H=\su{n}$, $\sp{n}\sp{1}$, $\spin{7}$ and $\sp{n}\u{1}$, respectively.

In all cases, the metrics $g_{t,s}$ are in one-to-one correspondence with the $G$-invariant metrics on $M$. In more detail, if $o\in M\cong G/H$ denotes the origin and $\lieP$ is identified with the tangent space of $M$ at $o$, recall from Section~\ref{SEC:HRF} that any $G$-invariant metric is determined by the $\Ad(H)$-invariant inner product $g(o)$ on $\lieP$. In order to describe all such inner products we will make use of an auxiliary bi-invariant metric $Q$ on $G$. Note that, apart from the case where $G = \sp{n+1}\sp{1}$, $Q$ is unique up to scaling.  In the exceptional case, it follows from \cite[Case (4) and Case (5), pg. 352 and 353]{Zi82} that the specific choice of $Q$ does not alter the set of $G$-invariant metrics on $G/H$, up to isometry.

The $\Ad(H)$-representation $\lieP$ splits orthogonally into two irreducible inequivalent subrepresentations $\lieP=\lieP_0\oplus\lieP_1$, where $\lieP_0$ is naturally identified with the vertical subspace $V$ of the Hopf fibration at $o$ and $\lieP_1$ with the horizontal subspace $H$. By Schur's Lemma any $\Ad(H)$-invariant inner product $\langle\cdot,\cdot\rangle$ on $\lieP$ is of the form $\langle\cdot,\cdot\rangle=x_0 Q(\cdot,\cdot)\vert_{\lieP_0}+x_1Q(\cdot,\cdot)\vert_{\lieP_1}$, for some $x_0,x_1\in\R_+$. Let $\langle\cdot,\cdot\rangle$ denote the $\Ad(H)$-invariant inner product which corresponds to the round metric of curvature one in the case of the spheres, and to the Fubini-Study metric in the case of $\C P^{2n+1}$. Then at the point $o$ we have the natural identification:
\begin{equation}\label{EQ:2-par-Diag}
g_{t,s}(\cdot,\cdot)=t\langle\cdot,\cdot\rangle\vert_V + s\langle\cdot,\cdot\rangle\vert_H.
\end{equation}
The Ricci tensor of the metric $g_{t,s}$ is diagonal, i.e. it is given by 
$$
\Ric(\cdot,\cdot)=r_t t\langle\cdot,\cdot\rangle\vert_V + r_s s\langle\cdot,\cdot\rangle\vert_H,
$$
for some $r_t,r_s\in\R$ depending on $(t,s)$. The values of $r_t,r_s$ can be deduced e.g. from the curvature computations in \cite{Zi82}. In the case of the spheres, they are:
$$
r_t=(d_v-1)\frac{1}{t} + d_H \frac{t}{s^2},\qquad r_s=\frac{d_H-1+3d_V}{s}-2d_V\frac{t}{s^2},
$$
where $d_V$ and $d_H$ are the dimensions of the vertical and horizontal subspaces of the corresponding Hopf fibrations. In the case of $\C P^{2n+1}$, the formulas can be deduced from \cite[pg. 357]{Zi82} (however, there appears to be a small error, as the formulas are interchanged):
$$
    r_t = \frac{4}{t} +\frac{4nt}{s^2},\qquad r_s = \frac{4n+8}{s}-\frac{4t}{s^2}.
$$
To describe the Ricci flow equation, note that the     decomposition $V\oplus H$ of a metric $g_{t,s}$ remains orthogonal under the flow because the Ricci flow preserves isometries. Thus, the diagonal form of $g_{t,s}$ and its Ricci tensor $\Ric$ implies that the Ricci flow of a metric $g_{t_0,s_0}$ is a family of metrics $g_{t(\ell),s(\ell)}$ such that the functions $t(\ell),s(\ell)$ satisfy the ODE system
\begin{equation}\label{EQ:RF_2param}
    \begin{cases}
      t'(\ell)=-2r_st(\ell),\\
      s'(\ell)=-2r_ts(\ell),\\
     (t(0),s(0))=(t_0,s_0), 
    \end{cases}
\end{equation}
with $\ell\in(L_-,L_+)$ for some $-\infty\leq L_- <0<L_+\leq\infty$. Also, note that both $r_t,r_s$ depend on $t(\ell)$ and $s(\ell)$.

We now describe in more detail the metrics $g_{x,y,z,s}$ on $\mathbb{S}^{4n+3}$, we refer to \cite{BPRZ21,Sb22,VZ09,Zi82} for details. These metrics all arise as $\sp{n+1}$-invariant metrics on $\sp{n+1}/\sp{n}=\mathbb{S}^{4n+3}$. To see this, endow $\sp{n+1}$ with some bi-invariant metric $Q$. The $\sp{n}$-representation $\lieP$ splits orthogonally into irreducible subrepresentations as $\lieP=\mathbb{R}\oplus \mathbb{R}\oplus\mathbb{R}\oplus\lieP_1$, where $\mathbb{R}$ denotes the trivial $\sp{n}$-representation.  As above, $\lieP_1$ is identified  with the horizontal subspace $H$, while $\mathbb{R}\oplus \mathbb{R}\oplus \mathbb{R}$ is identified with the vertical subspace $V$. We identify each $\mathbb{R}$ factor with the line spanned by $i,j,k\in\mathbb H$ respectively, so that $V$ is identified with $\operatorname{span}\{i,j,k\} = \operatorname{Im}(\mathbb{H})$. The basis vectors $i,j$, and $k$ are not orthogonal with respect to a general $\sp{n}$-invariant inner product on $\operatorname{Im}(\mathbb{H})$, but Ziller \cite[Case (4), pg. 352]{Zi82} showed using the action by $N(\sp{n})/\sp{n}$ that every $\sp{n}$-invariant inner product on $\operatorname{Im}(\mathbb{H})$ is isometric to one where $i$, $j$, and $k$ are orthogonal (but not necessarily of unit length). So, up to isometry, we may write a general metric on $G/H$ as 
\begin{equation}\label{EQ:4-par-Diag}
    g_{x,y,z,s}(\cdot,\cdot)=x\langle\cdot,\cdot\rangle\vert_{(i)} + y\langle\cdot,\cdot\rangle\vert_{(j)} + z\langle\cdot,\cdot\rangle\vert_{(k)} + s\langle\cdot,\cdot\rangle\vert_H.
\end{equation}
where $\langle \cdot,\cdot\rangle$ denotes the round metric of curvature one, and $(i),(j),(k)$ denote the subspaces spanned by $i,j,k$. We will later show that this form is invariant under the Ricci flow.  That is, we will show that if $i$, $j$, and $k$ start orthogonal, then they remain orthogonal for all positive time.

The Ricci tensor of the metric $g_{x,y,z,s}$ is also diagonal: 
$$
\Ric(\cdot,\cdot)=r_x x\langle\cdot,\cdot\rangle\vert_{(i)} + r_y y\langle\cdot,\cdot\rangle\vert_{(j)} + r_z z\langle\cdot,\cdot\rangle\vert_{(k)} + r_s s\langle\cdot,\cdot\rangle\vert_H,
$$
for some $r_x,r_y,r_z,r_s\in\R$ depending on $(x,y,z,s)$. Their values are known (see e.g. \cite{BPRZ21,Sb22,Zi82}):
\begin{align}\label{EQ:Ricci_eigen_4param}
\begin{split}
r_x &= \frac{4}{x}+\frac{4nx}{s^2}+2\frac{x^2-y^2-z^2}{xyz},\qquad\quad
r_y = \frac{4}{y}+\frac{4ny}{s^2}+2\frac{y^2-x^2-z^2}{xyz},\\
r_z &= \frac{4}{z}+\frac{4nz}{s^2}+2\frac{z^2-x^2-y^2}{xyz},\qquad\quad
r_s = \frac{4(n+2)}{s} -2\frac{x+y+z}{s^2}.
\end{split}
\end{align}
Let us describe the Ricci flow, we refer to \cite[Sections~2.1 and 4]{Sb22} for further details. First, as in the case of the metrics $g_{t,s}$, the decomposition $V\oplus H$ of the metrics $g_{x,y,z,s}$ remains orthogonal under the flow. Second, the decomposition $V=(i)\oplus(j)\oplus (k)$ also stays orthogonal. To show it, let $h_i$ be the endomorphism of $V\oplus H$ acting by rotation on $V$ by $\pi$ around the $i$-axis. That is 
$$
h_i(i)=i,\quad h_i(j)=-j,\quad h_i(k)=-k,
$$
and as the identity on $H$. We define $h_j$ and $h_k$ similarly. Note that all these maps can be realized via the conjugation action by $\mathrm{N(\sp{n})}/\sp{n}$ on $\sp{n+1}/\sp{n}$, so they are diffeomorphisms of $\sp{n+1}/\sp{n}$. By inspection, they are indeed isometries of any $\sp{n+1}$-invariant metric of the form $g_{x,y,z,s}$. Since the Ricci flow preserves isometries, it follows that $h_i, h_j,$ and $h_k$ remain isometries of the Ricci flow of $g_{x,y,z,s}$, denoted by $g(\ell)$. Then we have 
$$
g(\ell)(i, k)=g(\ell)(h_i(i), h_i(k))=g(\ell)(i, -k).
$$
This implies that $g(\ell)(i, k)=0$, and similarly, $g(\ell)(j, k)=g(\ell)(i, j)=0$. Moreover, since  $\Ric(g(\ell))$ is invariant under the isometries of $g(\ell)$, in a similar manner, we conclude that $\Ric(g(\ell))$ remains diagonal along the flow. Altogether, we conclude that the Ricci flow of a metric $g_{x_0,y_0,z_0,s_0}$ is the family of metrics $g_{x(\ell),y(\ell),z(\ell),s(\ell)}$ such that the functions $x(\ell),y(\ell),z(\ell),s(\ell)$ satisfy the ODE system
\begin{equation}\label{EQ:RF_4param}
    \begin{cases}
      x'(\ell)=-2r_x x(\ell),\\
      y'(\ell)=-2r_y y(\ell),\\
      z'(\ell)=-2r_z z(\ell),\\
      s'(\ell)=-2r_ts(\ell),\\
     (x(0),y(0),z(0),s(0))=(x_0,y_0,z_0,s_0),
    \end{cases}
\end{equation}
with $\ell\in(L_-,L_+)$ for some $-\infty\leq L_- <0<L_+\leq\infty$. Again, all of $r_x,r_y,r_z,r_s$ depend on $x(\ell)$, $y(\ell)$, $z(\ell)$ and $s(\ell)$.

\section{Two-parameter metrics}\label{sec:2param}

Here we prove the following result.

\begin{theorem}\label{THM:2-param}
    Let $g_{t,s}$ be a metric on a sphere $\mathbb{S}^{2n+1}$ or a complex projective space $\mathbb C P^{2n+1}$, $n\geq 1$, as constructed in Section~\ref{SEC:homog_metrics} out of any of the Hopf fibrations. If $g_{t,s}$ has $\sec>0$, then its Ricci flow has $\sec>0$ for all positive times of existence. 
\end{theorem}

Before giving the proof, we need to recall the subset $\mathcal P$ of metrics $g_{t,s}$ with $\sec>0$ on both spheres and complex projective spaces. For spheres, it is known that $g_{t,1}$ has $\sec >0$ if and only if $t<\frac{4}{3}$ (see e.g. \cite[Theorem~A]{VZ09}, where their metric $g_t$ corresponds to our metric $g_{t,1}$). We now show the corresponding result for $\mathbb{C}P^{2n+1}$.  Recall the metric $g^{\mathbb{C}P}_{t,s}$ from Section \ref{SEC:homog_metrics}.

\begin{proposition}\label{prop:cppos}
The metric $g^{\mathbb{C}P}_{t,1}$ has $\sec > 0$ if and only if $t < \frac{4}{3}$.
\end{proposition}

\begin{proof}
As mentioned above, the metric $g_{t,1}$ on $\mathbb{S}^{4n+3}$ has $\sec >0$ for $t < 4/3$. It follows from the Gray-O'Neill formulas \cite{Gr,ON} applied to the Riemannian submersion $(\mathbb{S}^{4n+3},g_{t,1}) \to(\mathbb{C}P^{2n+1},g^{\mathbb{C}P}_{t,1})$ that  $g^{\mathbb{C}P}_{t,1}$ has $\sec >0$ when $t < 4/3$.

Conversely, we know from \cite[pg.~357]{Zi82} that there are $2$-planes whose sectional curvature for the metric $g^{\mathbb{C}P}_{t,1}$ equals $4-3t$. In the notation of \cite[pg.~355]{Zi82}, these $2$-planes are spanned by the orthonormal vectors $Y_\alpha$ and $Y_{\alpha i}$, where $1\leq \alpha\leq n$ and $2\leq i \leq 3$, both of which lie in the horizontal space $H$ of the decomposition from \eqref{EQ:2-par-Diag}. Clearly, it follows that $g^{\mathbb{C}P}_{t,1}$ does not have $\sec>0$ for $t\geq \frac{4}{3}$.
\end{proof}

Summarizing, for either spheres or complex projective spaces, $g_{t,1}$ has $\sec >0$ if and only if $t<\frac{4}{3}$. Observe now that if we run the Ricci flow $g_{t(\ell),s(\ell)}$ of a metric $g_{t_0,1}$, then $s(t)$ may no longer be equal to 1, see \eqref{EQ:RF_2param}. Thus we need to characterize all metrics $g_{t,s}$ with $\sec>0$. To do so, note that $g_{t,s}$ is equal to $s g_{\frac{t}{s},1}$, so is a rescaling of $g_{\frac{t}{s},1}$. Define $F(t,s)=\frac{4}{3}-\frac{t}{s}$. Then it is clear that the subset $\mathcal P$ of metrics $g_{t,s}$ with $\sec>0$ is
\begin{align}\label{Eq:2-par-pos-cone}
 \mathcal P=\{g_{t,s}\; :\; F(t,s)>0\}.   
\end{align}

\begin{proof}[Proof of Theorem~\ref{THM:2-param}]
    Following Proposition~\ref{prop:RF_invariance}, we consider an arbitrary metric $g_{t,s}$ in $\partial\mathcal P$. From \eqref{Eq:2-par-pos-cone}, we obtain that $g_{t,s}\in \partial \mathcal P$ if and only if $g_{t,s}=sg_{\frac{4}{3}, 1}$ for some $s>0$. In fact, for our purposes we may assume that $s=1$ because of the following reasons. First, the condition $\sec>0$ is invariant under scalings. Second, it is a general fact that the Ricci flow of a scaled metric $sg$ is $sg\left(\frac{\ell}{s}\right)$, where $g(\ell)$ is the Ricci flow of $g$, see \cite[Section~3.1.3]{AH}. Altogether we conclude that a metric $sg\in\partial\mathcal P$ with $s>0$ satisfies the hypothesis in Proposition~\ref{prop:RF_invariance} if and only if $g$ does.

  Thus we consider $g_{\frac{4}{3},1}\in\partial\mathcal P$ and its Ricci flow $g(\ell)$, i.e. $g(0)=g_{\frac{4}{3},1}$. We shall show that
    \begin{equation}\label{eq:positive1}
        \frac{d}{d\ell}\Big\vert_{\ell=0} F(g(\ell))>0.
    \end{equation}
    Since $F(g(0))=0$, the positivity of the derivative implies that there is some small $\varepsilon>0$ such that $F(g(\ell))<0$ for all $\ell\in (-\varepsilon,0)$. Consequently, $F(g(\ell))\notin\mathcal P$ for all $\ell\in (-\varepsilon,0)$. Proposition~\ref{prop:RF_invariance} will then complete the proof.

    It remains to show that \eqref{eq:positive1} holds. To compute the derivative we use the chain rule:
    $$
    \frac{d}{d\ell}\Big\vert_{\ell=0} F(g(\ell)) =\nabla F(g(0))  g'(0), 
    $$ 
    where we interpret $\nabla F(g(0))$ as a $1\times 2$ vector and $g'(0)$ as a $2\times 1$ vector. Using the Ricci flow equation given in \eqref{EQ:RF_2param} we find that $g(\ell)$  is the family $g_{t(\ell),s(\ell)}$ determined by $(t',s')=(-2r_t t,-2r_s s)$. For $\ell=0$ we have $(t,s)=(\frac{4}{3},1)$, so $g'(0)=(-\frac{8}{3}r_t,-2r_s)^T$,  where $T$ denotes the transpose vector.

    The gradient of $F$ is $\nabla F=\left( -\frac{1}{s}, \frac{t}{s^2} \right)$, so $\nabla F(g(0))=\left( -1, \frac{4}{3} \right)$. Thus
$$
\frac{d}{d\ell}\Big\vert_{\ell=0} F(g(\ell))=\left( -1, \frac{4}{3} \right)\left(-\frac{8}{3}r_t,-2r_s\right)^T=\frac{8}{3}(r_t-r_s).
$$
For the case of a sphere, the Ricci eigenvalues of $g(0)=g_{\frac{4}{3},1}$ reduce to:
$$
r_t=\frac{3}{4}(d_V-1)+ \frac{4}{3} d_H,\qquad r_s=d_H -1 + \frac{1}{3} d_V.
$$
Then 
$$
\frac{d}{d\ell}\Big\vert_{\ell=0} F(g(\ell))=\frac{8}{3}(r_t-r_s)=\frac{2}{9} ( 5 d_V + 4 d_H + 3)>0.
$$
For complex projective spaces, the Ricci eigenvalues of $g(0)=g_{\frac{4}{3},1}$ reduce to 
$$
r_t = \frac{16}{3}n + 3,\qquad r_s=4n+\frac{8}{3}.
$$ 
Then 
$$
\frac{d}{d\ell}\Big\vert_{\ell=0} F(g(\ell)) = \frac{8}{3}(r_t-r_s) = \frac{8}{9}\left(4n+1\right) > 0.
$$
\end{proof}

\begin{remark}
  One can alternatively prove Theorem~\ref{THM:2-param} using the normalized Ricci flow. This reduces the case of $g_{t,s}$ to a one-parameter family of metrics. Since the normalized Ricci flow is the gradient flow of the scalar curvature functional, one would have to consider the function $\scal(g(t))$ and determine the sign of its derivative at the boundary metric.
\end{remark}

\section{Four-parameter metrics and a change of variables}\label{SEC:4param}

The goal of the rest of the article, except for the last section, is to prove the following result.

\begin{theorem}\label{THM:4-param}
    Let $g_{x,y,z,s}$ be a metric on $\mathbb{S}^{4n+3}$, $n\geq 1$, as constructed in Section~\ref{SEC:homog_metrics}. If $g_{x,y,z,s}$ has $\sec>0$, then its Ricci flow has $\sec>0$ for all positive times of existence. 
\end{theorem}

The proof of this theorem is rather involved. In this section we settle the strategy to follow, including references to results in later sections which will complete the proof. Also, at the end of this section we present a certain change of variables from $(y,z)$ to $(a,b)$ that will simplify the computations to come.

First we need to recall the space $\mathcal P$ of all metrics $g_{x,y,z,s}$ of $\sec>0$. This was determined by Verdiani and Ziller in \cite{VZ09} (note that they write $t_1,t_2,t_3$ instead of $x,y,z$). They focused on metrics of the form $g_{x,y,z,1}$, but this subset of $\mathcal P$ already determines $\mathcal P$ because $\sec>0$ is preserved under scalings. In order to state their result, we need some auxiliary functions. As in \cite{VZ09} we define the following functions of $\R_+^3$ in coordinates $(x,y,z)$:
\begin{align*}
H_1&=4-3x, & V_1 
= \frac{y^2+z^2-3x^2+2xy+2xz-2yz}{x},\\
H_2&=4-3y, &V_2 = \frac{x^2+z^2-3y^2+2yx+2yz-2xz}{y},\\
H_3&=4-3z, &V_3 = \frac{x^2+y^2-3z^2+2zx+2zy-2xy}{z}.\\
\end{align*}
We also define 
\begin{align*}
P_1 &= \sqrt{H_1 V_1} +yz - 3\vert yz-y-z+x\vert, \\
P_2 &= \sqrt{H_2 V_2} +xz - 3\vert xz-x-z+y\vert, \\
P_3 &= \sqrt{H_3 V_3} +xy - 3\vert xy-x-y+z\vert.
\end{align*}
With this notation, Verdiani and Ziller prove that a metric $g_{x,y,z,1}$ has $\sec >0$ if and only if the nine functions $H_i, V_i, P_i$, with $i=1,2,3$, are all positive. Note that the three inequalities $H_i>0$ simply mean that $x,y,z\in \left(0,\frac{4}{3}\right)$.

As we argued in Section~\ref{sec:2param}, we need to describe the set $\mathcal P$ of $\operatorname{sec}>0$ metrics for the whole family of metrics $g_{x,y,z,s}$, i.e., when $s$ is not necessarily equal to 1. To do this, note that we have a homothety:
$$
g_{x,y,z,s}\cong g_{\frac{x}{s},\frac{y}{s},\frac{z}{s},1}.
$$
Since $\sec > 0$ is preserved by scalings we naturally extend the nine functions $H_i,V_i,P_i:\R_+^3\to\R$ to functions $\tilde H_i,\tilde V_i,\tilde P_i:\R_+^4\to\R$ in the variables $(x,y,z,s)$ as follows:
$$
\tilde H_i(x,y,z,s):=H_i\left( \frac{x}{s},\frac{y}{s},\frac{z}{s}\right),\qquad \tilde V_i(x,y,z,s):=V_i\left( \frac{x}{s},\frac{y}{s},\frac{z}{s}\right),
$$
$$
\tilde P_i(x,y,z,s):=P_i\left( \frac{x}{s},\frac{y}{s},\frac{z}{s}\right),
$$ 
for $i=1,2,3$. Altogether we conclude that the subset $\mathcal P$ of metrics $g_{x,y,z,s}$ with $\sec>0$ is
$$
\mathcal P=\{g_{(x,y,z,s)}\; :\; \tilde H_i, \tilde{V}_i, \tilde{P}_i > 0 \text{ for all } i = 1,2,3\}.
$$
We are ready to present the strategy of the proof.

\begin{proof}[Strategy of the proof for Theorem~\ref{THM:4-param}]

We will prove Theorem \ref{THM:4-param} by appealing to Proposition \ref{prop:RF_invariance}. First we need to analyze the boundary $\partial \mathcal{P}$. As in the proof of Theorem~\ref{THM:2-param}, note that preservation of $\sec>0$ under the Ricci flow is not affected by scalings of the initial metric, so it is enough to study the behavior of metrics of the form $g_{x,y,z,1}$. Note that a metric $g_{x,y,z,1}$ lies in $\partial\mathcal{P}$ if and only if the nine functions $H_i, V_i, P_i$ for $i=1,2,3$ are non-negative and at least one of them is zero. 

Our main approach to showing that the hypotheses of Proposition \ref{prop:RF_invariance} are satisfied is similar in spirit to the proof of Theorem~\ref{THM:2-param}, although it is inevitably more involved since in this case the set $\mathcal P$ is determined by the positivity of nine functions instead of just one. Let us be more precise. As a first step we need to consider each metric of the form $g_{x,y,z,1}\in\partial\mathcal P$. Then at least one of the nine functions $H_i,V_i,P_i$ vanishes at $(x,y,z)$. Let $g(\ell)$ be the   Ricci flow of $g_{x,y,z,1}$. It will then be sufficient to show that, among the vanishing functions, we can find one, say $H_1$, which is smooth at $(x,y,z)$ and for which the corresponding extended function $\tilde{H}_1(g(\ell))$ has positive derivative at $\ell=0$. Because $\tilde H_1(g(0))=0$, the positivity of the derivative implies that there is some small $\varepsilon>0$ such that $\tilde H_1(g(\ell))<0$ for all $\ell\in (-\varepsilon,0)$. Hence, $\tilde H_1(g(\ell))\notin\mathcal P$ for all $\ell\in (-\varepsilon,0)$. Proposition~\ref{prop:RF_invariance} completes the proof.

Let us remark that, while the functions $H_i,V_i$ are smooth on the whole $\R_+^3$, it will turn out, however, that there are some metrics on the boundary $\partial\mathcal P$ at which some $P_i$ is non-differentiable. Fortunately, for these metrics it turns out that either $H_i$ or $V_i$ vanishes as well and the derivative of the corresponding extended function has the right sign, which is enough for our purposes.

In principle, one would have to study nine cases, in each of which one assumes that one of the nine functions vanishes while the others are non-negative. However, the functions $H_i$, $V_i$, and $P_i$ are symmetric in the three variables. This implies that, for instance, if a metric $g\in\partial\mathcal P$ satisfies $H_2=0$ or $H_3=0$, then there exists a metric $\overline{g}\in\partial\mathcal P$ with $H_1=0$ having the same Ricci flow behavior as $g$ with respect to the hypotheses of Proposition~\ref{prop:RF_invariance}; and analogous properties hold for $V_i$ and $P_i$. Therefore, it suffices to consider only three cases, corresponding to the vanishing of $H_1$, $V_1$, and $P_1$. For convenience, we define the sets
\begin{align*}
\mathcal B_H &:=\{g_{x,y,z,1}\in\partial \mathcal P\; :\; H_1(x,y,z)=0\}, \\
\mathcal B_V &:=\{g_{x,y,z,1}\in\partial \mathcal P\; :\; V_1(x,y,z)=0\}, \\
\mathcal B_P &:=\{g_{x,y,z,1}\in\partial \mathcal P\; :\; P_1(x,y,z)=0\}.
\end{align*}
Note that $\mathcal B_H$, $\mathcal B_V$ and $\mathcal B_P$ are not disjoint. By the discussion above it follows that Theorem~\ref{THM:4-param} will be proven if we show that for each metric $g_{x,y,z,1}\in\mathcal B_H\cup \mathcal B_V\cup\mathcal B_P$ it holds that the derivative of the extended function $\tilde H_1(g(\ell))$, $\tilde V_1(g(\ell))$ or $\tilde P_1(g(\ell))$ is positive for at least one of the functions $H_1$, $V_1$ or $P_1$ which vanishes at $(x,y,z)$. This plan will be carried out in Propositions~\ref{prop:H}, \ref{prop:V} and \ref{prop:P}, where we will consider, respectively, metrics in $\mathcal B_H$, $\mathcal B_V$ and $\mathcal B_P\setminus (\mathcal B_H \cup \mathcal B_V)$.
\end{proof}

To follow this strategy, the first task will be to describe the set $\mathcal B_H$, $\mathcal B_V$ or $\mathcal B_P$. It turns out that the computations get considerably tedious, so we settle for determining sets which are larger than $\mathcal B_H$, $\mathcal B_V$ or $\mathcal B_P$ but also simpler to describe. Knowing these larger sets will be enough for our purposes. We will describe these sets in the form $g_{x(y,z),y,z,1}$, i.e. with $(y,z)$ in $\R_+^2$ and with $x$ depending on $(y,z)$. In order to do so, we take advantage of two basic observations. The first one is stated in the following remark for later reference.

\begin{remark}\label{REM:square_0_43}
    Since $H_1$, $H_2$ and $H_3$ must be non-negative for all metrics $g_{x,y,z,1}$ in the closure of $\mathcal P$, it follows that $(x,y,z)\in (0,\frac{4}{3}]^3$. In particular, the same holds for any metric in $\mathcal B_H$, $\mathcal B_V$ or $\mathcal B_P$.
\end{remark}
\medskip

The second observation is that the formulas for $H_1$, $V_1$ and $P_1$ are symmetric in $y$ and $z$, so it will be convenient to express them in terms of elementary symmetric polynomials. To that end, we let 
\begin{equation}\label{EQ:change_variables}
 a:=\sigma_1(y,z) = y+z,\qquad b:=\sigma_2(y,z) = yz.   
\end{equation}
Since $(y,z)\in (0,\frac{4}{3}]^2$ by Remark~\ref{REM:square_0_43}, in order to work with the variables $a,b$ we need to determine the corresponding region $T$ in the $(a,b)$-plane.

\begin{proposition}\label{prop:feas1} 
The image  of $(0,\frac{4}{3}]^2$ under the function $\psi:\R_+^2\rightarrow \mathbb{R}^2$ defined by $(y,z)\mapsto (y+z,yz)$  is the region $T$ bounded by the curves $b = \frac{a^2}{4}$, $b = \frac{4}{3} (a-\frac{4}{3})$, and $b=0$, including the curves except that it excludes any point on $b=0$  (see Figure \ref{fig:feasible1}).
\end{proposition}

\begin{proof}We first observe that since $y,z\in \left(0,\frac{4}{3}\right]$, it easily follows that $a = y+z \in \left(0,\frac{8}{3}\right]$.  For each value of $a \in \left(0,\frac{8}{3}\right]$, we will determine the set of $b$-values which are obtainable with this choice of $a$.

Assume initially that $a\in  \left(0,\frac{4}{3}\right]$.  We observe that for any $y<a$, there is a $z\in \left(0,\frac{4}{3}\right]$ for which $y+z = a$, but that this statement is false for any $y\geq a$. Thus, for fixed $a\in \left(0,\frac{4}{3}\right]$, we conclude that $y\in (0,a)$.

Since $b= yz$ and $z= a-y$, we find $b = y(a-y)$ with $y\in (0,a)$. For fixed $a$, the formula $b= y(a-y)$ defines a parabola. The value of $b$ at $y=0$ and $y=a$ is $0$, while the value at the vertex $\frac{a}{2}\in (0,a)$ is $\frac{a^2}{4} > 0$.   It follows that, as $y$ varies over $(0,a)$, the value of $b$ varies in $\left(0,\frac{a^2}{4}\right]$.

We next assume that $a\in \left(\frac{4}{3},\frac{8}{3}\right]$.  Here, for a given $y\in \left(0,\frac{4}{3}\right]$, there is a $z\in \left(0,\frac{4}{3}\right]$ with $y+z = a$ if and only if $y\in \left[a-\frac{4}{3}, \frac{4}{3}\right]$.  Then we have $b = y(a-y)$ with $y\in\left[a-\frac{4}{3},\frac{4}{3}\right]$.  The value of $b$ at both endpoints $y=a-\frac{4}{3}$ and $y = \frac{4}{3}$ is $\frac{4}{3}\left(a-\frac{4}{3}\right)$, and the value at the vertex $\frac{a}{2}$  is $\frac{a^2}{4}$.  Since the coefficient of $y^2$ is negative, we deduce that $\frac{4}{3} \left(a-\frac{4}{3}\right) \leq \frac{a^2}{4}$ for $a\in\left(\frac{4}{3},\frac{8}{3}\right]$. It follows that, for $a\in\left(\frac{4}{3},\frac{8}{3}\right]$, the region $T$ ranges between the graphs of $b = \frac{4}{3}\left(a-\frac{4}{3}\right)$ and $b = \frac{a^2}{4}$.
\end{proof}

\begin{figure}
\begin{center}
\begin{tikzpicture}
\begin{axis}[
    axis lines=middle, xlabel={$a$}, ylabel={$b$},
    xmin=-0.2, xmax=3.2, ymin=-0.2, ymax=2.2,
    xtick={0,1,2,3}, ytick={0,1,2},
    axis line style={->}, axis on top
]
    \addplot[name path=para, thick, domain=0:2.666, samples=100] {x^2/4};
    \path[name path=bot] (axis cs:0,0) -- (axis cs:1.333,0) -- (axis cs:2.666, 1.777);
    \addplot[red!20] fill between[of=para and bot];
    \draw[thick] (axis cs:0,0) -- (axis cs:1.333,0) -- (axis cs:2.666, 1.777);
\end{axis}
\end{tikzpicture}
\caption{The region T}\label{fig:feasible1}
\end{center}
\end{figure}
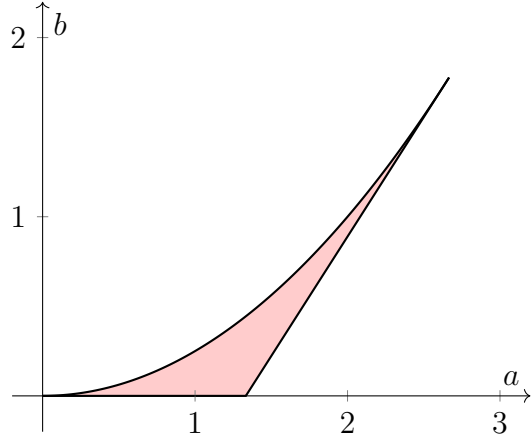

Lastly, we express $V_1$, $P_1$, $r_x$ and $r_s$, all of which symmetric in $y$ and $z$, in terms of $(x,a,b,s)$. Note that, in contrast, $r_y$ and $r_z$ are not symmetric in $y$ and $z$.

\begin{lemma}\label{lem:convertab} 
In terms of $a$ and $b$, we have 
$$
V_1 = \frac{a^2-4b+2xa - 3x^2}{x}\qquad \text{ and }\qquad P_1 = \sqrt{H_1 V_1} + b -3|b-a+x|.
$$  
In addition, the Ricci eigenvalues $r_x$ and $r_s$ are given by 
$$
r_x = \frac{4}{x} + \frac{4nx}{s^2} + 2 \frac{x^2-a^2 + 2b}{bx}\qquad \text{ and }\qquad r_s =\frac{4(n+2)}{s} - 2 \frac{x+a}{s^2}.
$$

\end{lemma}

\section{\texorpdfstring{First case: $H_1=0$}{First case: H1=0}}\label{sec:H}

Recall the set $\mathcal B_H:=\{g_{x,y,z,1}\in\partial \mathcal P\; :\; H_1(x,y,z)=0\}$ defined in Section~\ref{SEC:4param}. In this section we prove the following proposition.

\begin{proposition}\label{prop:H}
Suppose $g_{x,y,z,1}\in\mathcal B_H$, and let $g(\ell)$ denote the Ricci flow of $g(0)=g_{x,y,z,1}$, with $\ell\in (L_-,L_+)$ and $-\infty\leq -L<0<L_+\leq\infty$. Then there is some $\varepsilon>0$ such that $g(\ell)\notin\mathcal P$ for all $\ell\in (-\varepsilon,0)$.
\end{proposition}

We observe that, since $(y,z)\in \left(0,\frac{4}{3}\right]^2$ for all $g_{x,y,z,1}\in \mathcal B_H$ by Remark~\ref{REM:square_0_43}, and since $H_1(x,y,z)=0$ if and only if $x=\frac{4}{3}$, we have the obvious inclusion
$$
\mathcal B_H\subset\left\{g_{\frac{4}{3},y,z,1}\; :\; (y,z)\in \left(0,\tfrac{4}{3}\right]^2 \right\}.
$$
Recall that $\tilde H_1(x,y,z,s)=4-3\frac{x}{s}$, so it is smooth on the whole $\R_+^4$, and in particular at any point of $\left\{g_{\frac{4}{3},y,z,1}\; :\; (y,z)\in \left(0,\frac{4}{3}\right]^2 \right\}$. Then Proposition \ref{prop:H} follows easily from the following proposition.

\begin{proposition}\label{prop:Hprime}
    For all $(y,z)\in \left(0,\frac{4}{3}\right]^2$ it holds that $\frac{d}{d\ell}\Big\vert_{\ell=0} \tilde H_1(g(\ell)) > 0$, where $g(\ell)$ denotes the Ricci flow of $g(0)=g_{\frac{4}{3},y,z,1}$.
\end{proposition}

Throughout the proof we will make the change of variables from $(y,z)$ to $(a,b)$ given in \eqref{EQ:change_variables}. It will be convenient to define a function $F^H:\mathbb{R}^2\rightarrow \mathbb{R}$ by the formula 
$$
F^H(a,b) = \frac{16}{3} - 3a^2 + 4b+ 4ab.
$$
The map $F^H$ will arise in the proof of the proposition, and we will use the fact that it is positive on the region $T$ given in Proposition~\ref{prop:feas1}. 

\begin{lemma}\label{LEM:FH_positive}
For any $(a,b)\in T$ it holds that $F^H>0$.
\end{lemma}

We postpone the proof of this lemma until the end of this section and proceed with the proof of Proposition~\ref{prop:Hprime}.

\begin{proof}[Proof of Proposition~\ref{prop:Hprime} assuming Lemma~\ref{LEM:FH_positive}]
Take an arbitrary metric $g_{\frac{4}{3},y,z,1}$ with $(y,z)\in\left(0,\frac{4}{3}\right]^2$, and let $g(\ell)$ denote its Ricci flow, i.e. $g(0)=g_{\frac{4}{3},y,z,1}$. To compute the derivative we use the chain rule:
    $$
    \frac{d}{d\ell}\Big\vert_{\ell=0} \tilde H_1(g(\ell)) =\nabla \tilde H_1(g(0)) g'(0), 
    $$ 
    where, analogously to the proof of Theorem \ref{THM:2-param}, we interpret $\nabla\tilde H_1(g(0))$ as a $1\times 4$ vector and $g'(0)$ as a $4\times 1$ vector. From the Ricci flow equation \eqref{EQ:RF_4param} we know that $g(\ell)$ is the family $g_{x(\ell),y(\ell),z(\ell),s(\ell)}$ determined by
$$
(x',y',z',s')=(-2r_x x,-2r_y y,-2r_z z, -2r_s s).
$$
For $\ell=0$ we have $(x,y,z,s)=(\frac{4}{3},y,z,1)$, so $g'(0)=(-\frac{8}{3}r_x,-2r_y y,-2r_z z,-2r_s)^T$.

The gradient of $\tilde H_1$ is $\nabla\tilde H_1 = \left( -\frac{3}{s}, 0, 0, \frac{3x}{s^2} \right)$, which at $g_{\frac{4}{3},y,z,1}$ equals $\nabla\tilde H_1(g(0)) = \left( -3, 0, 0, 4 \right)$. Then:
$$
\frac{d}{d\ell}\Big\vert_{\ell=0} \tilde H_1(g(\ell))=\left( -3, 0, 0, 4 \right)
\left(  -\frac{8}{3}r_x, -2r_y y,  -2r_z z, -2r_s \right)^T=8(r_x-r_s).
$$
In order to analyze the sign of this expression we use the variables $(a,b)\in T$ from \eqref{EQ:change_variables} instead of $(y,z)\in \left(0,\frac{4}{3}\right]^2$. The Ricci eigenvalues $r_x,r_s$ are given in Lemma \ref{lem:convertab}. At the metric $g_{\frac{4}{3},y,z,1}$, they equal
$$
r_x = 3+\frac{16n}{3}+3\frac{\frac{16}{9}-a^2 + 2b}{2b},\qquad r_s = 4(n+2) -\frac{8}{3}-2a.
$$
After substituting, we get
$$
\frac{d}{d\ell}\Big\vert_{\ell=0} \tilde H_1(g(\ell))=8(r_x-r_s)=\frac{4}{b}\left( \frac{16}{3} - 3a^2 + \left(\frac{4}{3} + \frac{8}{3}n\right)b+4ab\right).
$$

Now, by Lemma~\ref{LEM:FH_positive}, we know that $F^H(a,b)>0$ for all $(a,b)\in T$. Multiplying by $\frac{4}{b} > 0$, we get the following chain of inequalities on $T$
\begin{align*} 
0& <  \frac{4}{b}F^H(a,b)\\
&= \frac{4}{b}\left(\frac{16}{3}-3a^2 + 4b+4ab\right)\\
&= \frac{4}{b}\left( \frac{16}{3} - 3a^2 + \left(\frac{4}{3}+ \frac{8}{3}(1)\right)b + 4ab\right)\\
&\leq\frac{4}{b}\left( \frac{16}{3} - 3a^2 + \left(\frac{4}{3} + \frac{8}{3}n\right)b+4ab\right)\\
&= \frac{d}{d\ell}\Big\vert_{\ell=0}\tilde H_1(g(\ell)).
\end{align*} 
It follows that $\frac{d}{d\ell}\Big\vert_{\ell=0}\tilde H_1(g(\ell)) > 0$ on $T$.
\end{proof}

We now prove that $F^H>0$ on $T$.

\begin{proof}[Proof of Lemma~\ref{LEM:FH_positive}]
We will show that the graph determined by the equation $F^H = 0$ does not intersect the region $T$. As $T$ is obviously connected and $F^H$ is continuous on $\R^2$, it will follow that the sign of $F^H$ is constant on $T$. Taking the sample point $\left(1,\frac{1}{8}\right)\in T$ we see that $F^H\left(1,\frac{1}{8}\right) = \frac{10}{3} > 0$, so this will complete the proof that $F_H > 0$ on $T$.

Hence let us show that $F^H = 0$ does not intersect $T$. The equation $F^H=0$ is linear in $b$.  Solving for $b$, we find $$b = \frac{9a^2-16}{12(a+1)}.$$

We first observe that for $a \leq \frac{4}{3}$ this formula gives $b\leq 0$, so $(a,b)\notin T$. Thus, any point in the intersection of $T$ and the graph of $F^H=0$ must have $a \in \left(\frac{4}{3},\frac{8}{3}\right]$. For these values of $a$, the corresponding $b$-values satisfy $\frac{4}{3}\left(a-\frac{4}{3}\right)\leq b\leq \frac{a^2}{4}$. Hence, to see that the graph of $F_H=0$ does not intersect the region $T$, it is sufficient to show that $\frac{9a^2-16}{12(a+1)}< \frac{4}{3}\left(a - \frac{4}{3}\right)$ for all $a\in\left(\frac{4}{3}, \frac{8}{3}\right]$. This holds if and only if $9a^2 - 16 < 16a^2-\frac{16}{3}a -\frac{64}{3}$, which holds if and only if $0 < 7a^2 -\frac{16}{3}a -\frac{16}{3} = \frac{1}{3}(7a+4)(3a-4)$. For $a > \frac{4}{3}$ both factors $(7a+4)$ and $(3a-4)$ are positive, which implies the desired inequality.
\end{proof}

\section{\texorpdfstring{Second case: $V_1=0$}{Second case: V1=0}}\label{sec:V}

Recall the set $\mathcal B_V:=\{g_{x,y,z,1}\in\partial \mathcal P\; :\; V_1(x,y,z)=0\}$ defined in Section~\ref{SEC:4param}. In this section we prove the following proposition. 

\begin{proposition}\label{prop:V}
Suppose $g_{x,y,z,1}\in\mathcal B_V$, and let $g(\ell)$ denote the Ricci flow of $g(0)=g_{x,y,z,1}$, with $\ell\in (L_-,L_+)$ and $-\infty\leq -L<0<L_+\leq\infty$. Then there is some $\varepsilon>0$ such that $g(\ell)\notin\mathcal P$ for all $\ell\in (-\varepsilon,0)$.
\end{proposition}

First, we need to study the set $\mathcal B_V$. This case is more complicated than that of $\mathcal B_H$ from Section~\ref{sec:H} and it will take some work to describe a certain set containing $\mathcal B_V$.  This will be achieved in \eqref{EQ:set_BV} below. To do so, we already switch to the variables $(a,b)$ from \eqref{EQ:change_variables} instead of $(y,z)$, although we continue to write the metrics $g_{x,y,z,1}$ in terms of the variables $y,z$. Throughout this section we will use that $(a,b)\in T$ for any metric in $\mathcal B_V$, see Remark~\ref{REM:square_0_43} and Proposition~\ref{prop:feas1}. Recall from Lemma \ref{lem:convertab} that 
$$
V_1 = \frac{a^2-4b+2xa - 3x^2}{x}.
$$  
Suppose $g_{x,y,z,1}$ satisfies $V_1=0$.  Then the numerator, which is quadratic in $x$, must vanish.  It follows that $x$ must be equal to one of $x_\pm$, where
\begin{equation}\label{eq:x_plus}
    x_\pm = \frac{a \pm 2\sqrt{a^2 -3b}}{3}.
\end{equation}
The first observation is that $x_-$ cannot occur in our setting.

\begin{lemma}\label{LEM:no_x_minus}
There is no metric of the form $g_{x,y,z,1}$ with $x=x_-$ in $\mathcal{B}_V$.
\end{lemma}

\begin{proof} 
Note that $x_-(a,b)$ defines a continuous function on the region $b\leq\frac{a^2}{3}$, which contains the region $T$. We will show that the graph of $x_-(a,b) = 0$ only intersects $T$ along the boundary and that $x_- < 0$ at one point in the interior of $T$. This implies that $x_-\leq 0$ on $T$. That is, $x_-\notin \left(0,\frac{4}{3}\right]$, hence Remark~\ref{REM:square_0_43} implies that $g_{x_-,y,z,1}\notin \mathcal{B}_V$.

As a sample point in $T$ we take $(a,b) = \left(1,\frac{5}{27}\right)$, where $\sqrt{a^2-3b}=\frac{2}{3}$. We compute $x_-\left(1,\frac{5}{27}\right)=-\frac{1}{9}<0$.

Now we study the intersection of the set $x_-(a,b)=0$ and $T$. Note that the equation $x_- = 0$ is equivalent to $a = 2\sqrt{a^2 - 3b}$, i.e., they have precisely the same solution set. Because $a>0$ on $T$, it is equivalent to $a^2 = 4(a^2 - 3b)$. This simplifies to $b= \frac{a^2}{4}$, which is a boundary component of $T$. This completes the proof.
\end{proof}

Thus, metrics in $\mathcal B_V$ must be of the form $g_{x,y,z,1}$ with $(y,z)\in \left(0,\frac{4}{3}\right]^2$ and $x=x_+$ given by the formula in \eqref{eq:x_plus}. We deduce that in this situation we have the obvious inclusion
\begin{equation}\label{EQ:set_BV}
\mathcal B_V\subset\left\{g_{x_+,y,z,1}\; :\; (y,z)\in \left(0,\tfrac{4}{3}\right]^2 \right\}.
\end{equation}
Clearly $V_1$ is smooth on $\R_+^3$. Since $\tilde V_1=\frac{V_1}{s}$, it follows that $\tilde V_1$ is smooth on $\R_+^4$, and in particular at any point of $\{g_{x_+,y,z,1}\; :\; (y,z)\in \left(0,\frac{4}{3}\right]^2 \}$. Proposition~\ref{prop:V} now follows easily from the following proposition.

\begin{proposition}\label{prop:Vprime}  
    For all $(y,z)\in \left(0,\frac{4}{3}\right]^2$ it holds that $\frac{d}{d\ell}\Big\vert_{\ell=0} \tilde V_1(g(\ell)) > 0$, where $g(\ell)$ denotes the Ricci flow of $g(0)=g_{x_+,y,z,1}$.
\end{proposition}

Towards proving this proposition, we now define, for each $n\in \mathbb{N}$, the function 
$$
F^V_n:\{(a,b)\in \mathbb{R}^2 \;:\; a^2 - 3b\geq 0\}\rightarrow \mathbb{R}
$$
given by 
\begin{align*}
F_n^V(a,b) =&n\left(-8b\bigl(\sqrt{a^2-3b}-a\bigr)^2\bigl(2\sqrt{a^2-3b}+a\bigr)^2 \right)\\
    &+\sqrt{a^2 - 3b}(-32a^3 + 96ab) + 32a^4 - 144a^2b.    
\end{align*}
We record several observations about $F^V_n$.  First, since $b\leq  \frac{a^2}{4} < \frac{a^2}{3}$ on $T$, it follows that $a^2 - 3b > 0$ on $T$, so each $F_n^V$ defines a continuous function when restricted to $T$.  Second, since $b > 0$ on $T$, the coefficient of $n$ is non-positive, which implies that $F_n(a,b)\leq F_1(a,b)$ for all $(a,b)\in T$.  Third, it can be checked that
$$
F_1^V(a,b) = -8 \left(2 \sqrt{a^{2}-3 b}+a\right) \left(a-\sqrt{a^{2}-3 b}\right)^{2} \left(2 b \sqrt{a^{2}-3 b}+a \left(b+2\right)\right),
$$ 
which is strictly negative on $T$.  Together, these imply that $F_n(a,b) < 0$ for all $(a,b)\in T$. Let us state it for later reference.

\begin{lemma}\label{LEM:FV_negative}
    For any $(a,b)\in T$ and $n\in\mathbb N$ it holds that $F_n(a,b) < 0$.
\end{lemma}

We are now in a position to prove Proposition \ref{prop:Vprime}.

\begin{proof}[Proof of Proposition \ref{prop:Vprime}]
Take an arbitrary metric $g_{x_+,y,z,1}$ with $(y,z)\in\left(0,\frac{4}{3}\right]^2$, and let $g(\ell)$ denote its Ricci flow, i.e. $g(0)=g_{x_+,y,z,1}$. As in the proof of Proposition~\ref{prop:Hprime}, to compute the derivative of $\tilde V_1(g(\ell))$ at $\ell=0$  we use the chain rule:
\begin{equation}\label{EQ:chain_for_V1}
        \frac{d}{d\ell}\Big\vert_{\ell=0}  \tilde V_1(g(\ell)) =\nabla  \tilde V_1(g(0)) g'(0). 
\end{equation}
It follows from the  Ricci flow equation \eqref{EQ:RF_4param}  that $g
(\ell)$ is the family $g_{x(\ell),y(\ell),z(\ell),s(\ell)}$ determined by
$$(x',y',z',s')=(-2r_x x,-2r_y y,-2r_z z, -2r_s s).$$
At $\ell=0$ we have $(x,y,z,s)=(x_+,y,z,1)$, so $g'(0)=(-2r_x x_+,-2r_y y,-2r_z z,-2r_s)^T$.

Now we compute the gradient of $\tilde V_1$. Observe that $\tilde V_1=\frac{V_1}{s}$. Then the gradient is
$$
\nabla\tilde V_1=\left( \frac{1}{s}\frac{\partial V_1}{\partial x}  , \frac{1}{s}\frac{\partial V_1}{\partial y} , \frac{1}{s}\frac{\partial V_1}{\partial z} , \frac{1}{s^2}\left(s\frac{\partial V_1}{\partial s}-V_1    \right)    \right)
$$
Note that $\frac{\partial V_1}{\partial s}=0$ because $V_1$ does not depend on $s$. Also, at $g(0)=g_{x_+,y,z,1}$ we have $s=1$ and $V_1=0$. Thus, the gradient reduces to
\begin{align}
\nabla\tilde V_1(g(0))&=\left( \frac{\partial V_1}{\partial x}  , \frac{\partial V_1}{\partial y} , \frac{\partial V_1}{\partial z} , 0  \right)\Big\vert_{g(0)} \nonumber\\
& =\left( \frac{(-6x+2y+2z)x-xV_1}{x^2}, \frac{2(y+x-z)}{x} , \frac{2(z+x-y)}{x},0 \right)\label{EQ:gradient_V} \\
& = \frac{2}{x}\left( -3x+y+z, y+x-z , z+x-y,0 \right).\nonumber
\end{align}
Inserting the corresponding values in \eqref{EQ:chain_for_V1} we find that:
\begin{equation}\label{eq:deriv_V1}
    \frac{d}{d\ell}\Big\vert_{\ell=0} \tilde V_1(g(\ell))= -\frac{4}{x}\Bigl( (-3x+y+z)r_x x + (y+x-z)r_y y + (z+x-y)r_z z    \Bigr).
\end{equation}
Observe that in the equations \eqref{EQ:gradient_V} and \eqref{eq:deriv_V1} we have dropped the subindex $+$ from $x_+$ in order to simplify the notation.

Now we make the following changes in \eqref{eq:deriv_V1}. First, we substitute $r_x,r_y,r_z$ by their values from \eqref{EQ:Ricci_eigen_4param} at the metric $g_{x,y,z,1}$. This gives an expression in the variables $x,y,z$ which is symmetric in $y,z$. Second, we make the change of variables from $(y,z)\in \left(0,\frac{4}{3}\right]^2$ to $(a,b)\in T$, thus obtaining an expression in the variables $x,a,b$. Third and last, we substitute $x$ by $x_+=\frac{a + 2\sqrt{a^2 -3b}}{3}$, which gives an expression in $a,b$. These tedious computations give the following identity:
$$
\frac{d}{d\ell}\Big\vert_{\ell=0} \tilde V_1(g(\ell)) = -\frac{4}{x_+}\left( \frac{F_n^V(a,b)}{9b\left(a + 2\sqrt{a^2-3b}\right)}\right).
$$  
Since $-\frac{4}{x_+} < 0$ and both $a$ and $b$ are positive, it now easily follows that $\frac{d}{d\ell}\Big\vert_{\ell=0} \tilde V_1(g(\ell)) > 0$ if and only if $F_n^V(a,b) <0$ for all $(a,b)\in T$. The latter holds by Lemma~\ref{LEM:FV_negative}.
\end{proof}

\section{\texorpdfstring{Third case: $P_1=0$}{Third case: P1=0}}\label{SEC:P1_0}

Recall the set $\mathcal B_P:=\{g_{x,y,z,1}\in\partial \mathcal P\; :\; P_1(x,y,z)=0\}$ defined in Section~\ref{SEC:4param}. Over this and the next section, we prove the following proposition. 

\begin{proposition}\label{prop:P}
Suppose $g_{x,y,z,1}\in\mathcal B_P$, and let $g(\ell)$ denote the Ricci flow of $g(0)=g_{x,y,z,1}$, with $\ell\in (L_-,L_+)$ and $-\infty\leq -L<0<L_+\leq\infty$. Then there is some $\varepsilon>0$ such that $g(\ell)\notin\mathcal P$ for all $\ell\in (-\varepsilon,0)$.
\end{proposition}

Observe that it is enough to prove Proposition~\ref{prop:P} for metrics in $\mathcal B_P\setminus (\mathcal B_H\cup\mathcal B_V)$, since the metrics in $\mathcal B_H$ or $\mathcal B_V$ have already been considered in Propositions~\ref{prop:H} and \ref{prop:V}, respectively. With this in mind, Proposition~\ref{prop:P} will ultimately follow from Proposition~\ref{prop:Pprime}. 

The study of the set $\mathcal B_P$ is significantly more complicated than that of $\mathcal B_H$ or even $\mathcal B_V$ from Sections~\ref{sec:H} and \ref{sec:V}. The main results of this section are Proposition~\ref{prop:possible_x4} and Corollary~\ref{COR:BP_minus_BH}, where we give an explicit set containing $\mathcal B_P\setminus (\mathcal B_H \cup \mathcal B_V)$. This larger set will be called $\hat S$ in Section~\ref{SEC:derivative_of_P1}.

In order to study $\mathcal B_P$, it is convenient to switch  from the variables $(y,z)$, which belong to $(0,4/3]^2$ by Remark~\ref{REM:square_0_43}, to the corresponding variables $(a,b)=(x+y,xy)$ belonging to the region $T$ from Proposition~\ref{prop:feas1}. Nonetheless, we continue to write the metrics $g_{x,y,z,1}$ in terms of the variables $y, z$. From Lemma~\ref{lem:convertab} we know that
$$
P_1 = \sqrt{H_1 V_1} +b - 3\vert b-a+x\vert.
$$
We consider three cases depending on the sign of the expression inside the absolute value:
$$
b-a+x=0,\qquad b-a+x>0,\qquad b-a+x<0.
$$
We begin by showing that the first case is impossible.

\begin{lemma}\label{LEM:absvaluenonzero}
There are no metrics of the form $g_{x,y,z,1}$ in $\mathcal B_P$ with $b-a+x=0$.
\end{lemma}

\begin{proof}
Assume for a contradiction that $b-a+x = 0$. Substituting this into the equation $P_1 = 0$, we find $\sqrt{H_1 V_1} + b = 0$. This is a contradiction as the left hand side is positive, being a sum of a non-negative term and a positive term.
\end{proof}

The cases $b-a+x>0$ and $b-a+x<0$ are significantly more complicated and shall be considered in Subsections~\ref{SUBSEC:bax_pos} and \ref{SUBSEC:bax_neg}, respectively.

\subsection{\texorpdfstring{The case where $b-a+x>0$}{The case where b-a+x>0}}\label{SUBSEC:bax_pos}

The main result of this subsection is the following.

\begin{lemma}\label{LEM:no_bax_positive}
There are no metrics of the form $g_{x,y,z,1}$ in $\mathcal B_P\setminus(\mathcal B_H\cup \mathcal B_V)$ with $b-a+x>0$. 
\end{lemma}

Lemma~\ref{LEM:no_bax_positive} will follow from Lemmas~\ref{LEM:A1zero} and \ref{LEM:A1nonZero} below. Before stating them, we need to collect some observations and introduce some notation.

The assumption $b-a+x>0$ implies that the equation $P_1=0$ is equivalent to $\sqrt{H_1 V_1} = 2b-3a+3x$, so there are solutions only if $2b-3a+3x\geq 0$.  However, since $H_1$ and $V_1$ are non-zero for any $g_{x,y,z,1}\in \mathcal{B}_P\setminus (\mathcal{B}_H\cup\mathcal{B}_V)$, we conclude that for such a metric:
\begin{equation}\label{EQ:eq_on_xyz_ab}
    2b-3a+3x  >  0.
 \end{equation}
Under the assumption in~\eqref{EQ:eq_on_xyz_ab}, the equation $\sqrt{H_1 V_1} = 2b-3a+3x$  implies  $H_1 V_1 = (2b-3a+3x)^2$. Expanding and rearranging this equation we find that it is equivalent to the following:
\begin{equation}\label{EQ:Quad_P=0}
3 \left(a-b-1\right) x^{2}+\left(-3 a^{2}+\left(3 b+2\right) a-b^{2}+3 b\right) x+a^{2}-4 b=0.
\end{equation}
The left-hand side of \eqref{EQ:Quad_P=0} is a polynomial $Ax^2+Bx+C$, where:
\begin{align*}
A&:=3 \left(a-b-1\right), \\
B&:=-3 a^{2}+\left(3 b+2\right) a-b^{2}+3 b,\\
C&:=a^{2}-4 b.
\end{align*}
We summarize the previous discussion in the following statement, for later reference.

\begin{lemma}\label{LEM:conditions_P1_0_case1}
    A metric $g_{x,y,z,1}$  with $b-a+x>0$ belongs to $\mathcal{B}_P\setminus (\mathcal{B}_H\cup \mathcal{B}_V)$ if and only if the following two conditions are satisfied:
    \begin{itemize}
        \item $2b-3a+3x  > 0$, and 
        \item $Ax^2+Bx+C=0$.
    \end{itemize}
\end{lemma}

In order to analyze the solutions of $Ax^2+Bx+C=0$, we divide our arguments into two cases: $A=0$ and $A\neq 0$, and show that neither of these cases can occur.

\begin{lemma}\label{LEM:A1zero}
There are no metrics of the form $g_{x,y,z,1}$ in $\mathcal B_P\setminus(\mathcal B_H  \cup \mathcal{B}_V)$ with $(a,b)\in T\cap\{A=0\}$ and $b-a+x>0$.
\end{lemma}

\begin{proof}
   We argue by contradiction, so suppose there exists such a metric. By Lemma~\ref{LEM:conditions_P1_0_case1}, the parameter $x$ must satisfy $Bx+C=0$. We distinguish two cases: $B=0$ and $B\neq 0$. If $B=0$, then $C=0$. Substituting $b=a-1$ (which corresponds to $A=0$) into $C=0$ we find that $a=2$, and consequently $b=1$. Using~\eqref{EQ:eq_on_xyz_ab} we find that $x\geq\frac{4}{3}$ (or equivalently $H_1\leq 0$), which yields a contradiction because any metric in $\mathcal B_P\setminus(\mathcal B_H\cup\mathcal B_V)$ must satisfy $H_1 > 0$.

    If $B\neq 0$, then $x=-\frac{C}{B}$. Substituting $b=a-1$ into this expression we obtain $x=1$. However, for $x=1$ we find that $b-a+x=0$, which is a contradiction with the assumption $b-a+x>0$.
\end{proof}

Now we assume $A\neq 0$, so the solutions of \eqref{EQ:Quad_P=0} correspond to the roots of a quadratic polynomial in the variable $x$ which has solutions:
$$
x_1=\frac{-B+\sqrt{\Delta}}{2A},\qquad x_2=\frac{-B-\sqrt{\Delta}}{2A},
$$
where $\Delta:=B^2-4AC$. We now show that both formulas define continuous functions on $T\cap \{A\neq 0\}$.

\begin{lemma}\label{LEM:x1_x42_continuous_on_T-A}
For any $(a,b)\in T$ it holds that $\Delta \geq 0$. In particular, the formulas for $x_1$ and $x_2$ define continuous functions on $T\cap \{A\neq 0\}$.
\end{lemma}

\begin{proof}
We will show that the absolute minimum of $\Delta$ on the closure of $T$ is $0$.  This absolute minimum must either occur on the boundary of $T$ or it must be a critical point. We begin by analyzing the behavior of $\Delta$ on the boundary of $T$.

Along the boundary curve $b=\frac{a^2}{4}$, we find $\Delta\left(a,\frac{a^2}{4}\right) =  \frac{1}{256} a^2 (a-8)^2(a-2)^4$, which is obviously non-negative.  Along the boundary curve $b=0$, we find $\Delta(a,0) = a^2(-4 + 3a)^2$ which is also obviously non-negative.  Along the boundary curve $b = \frac{4}{3}\left(a-\frac{4}{3}\right)$, we find $\Delta\left(a, \frac{4}{3}\left(a-\frac{4}{3}\right)\right) =\frac{\left(441 a^{2}-2040 a+2368\right) \left(-4+3 a\right)^{2}}{6561}$. Since the factor $(441 a^{2}-2040 a+2368)$ has no real roots, we conclude that $\Delta \geq 0$ on this boundary curve as well.

Having shown that $\Delta \geq 0$ on the boundary of $T$, we now turn our attention towards interior critical points.  We will show that $(a,b)=\left( \frac{17}{6} - \frac{\sqrt{465}}{18}, \frac{23}{3}-\frac{\sqrt{465}}{3}\right)$ is the  only possible critical point in the interior of $T$. Believing this momentarily, evaluating $\Delta$ at this point we get 
$$
\frac{17825\sqrt{465}}{162} - \frac{128125}{54} = \frac{17825 \sqrt{\frac{236391585}{508369}}}{162} - \frac{128125}{54} > \frac{17825 \sqrt{\frac{236390625}{508369}}}{162} - \frac{128125}{54} = 0,
$$ 
which implies $\Delta \geq 0$ on $T$, as claimed.

Now let us analyze the interior critical points of $\Delta$, i.e. the common zeros $(a_0,b_0)$ of the partial derivatives $\Delta_a$ and $\Delta_b$. Recall from Section \ref{SEC:resultant_Sturm} that both of the resultants $\operatorname{res}_a(\Delta_a,\Delta_b)\in\mathbb Z[b]$ and $\operatorname{res}_b(\Delta_a,\Delta_b)\in\mathbb Z[a]$ must vanish at $(a_0,b_0)$. Computing $\Delta_a$ and $\Delta_b$ and the corresponding resultants one finds that:
$$
\operatorname{res}_b:=\operatorname{res}_b(\Delta_a,\Delta_b) = -2304(a-2)^4(3a-4)(27a^2-153a+178)
$$ 
and 
$$
\operatorname{res}_a:=\operatorname{res}_a(\Delta_a,\Delta_b) = 62208b(b-1)^4(9b^2 - 138b + 64).
$$  
The roots of $\operatorname{res}_b$ are $2,\frac{4}{3}, \frac{17}{6} - \frac{\sqrt{465}}{18}$, and $\frac{17}{6} + \frac{\sqrt{465}}{18}$, but the last root is larger than $\frac{16}{6} = \frac{8}{3}$, so lies outside of $T$.  Since $b > 0$ on the interior of $T$, the relevant roots of $\operatorname{res}_a$ are $1, \frac{23}{3} - \frac{\sqrt{465}}{3},$ and $\frac{23}{3} + \frac{\sqrt{465}}{3}$, where, similarly, the last root is larger than $ \frac{16}{9}$, so lies out of $T$. In summary, any interior critical point $(a_0,b_0)$ must have 
\begin{equation}\label{EQ:possibilities_for_a_b}
    a_0\in\left\{2,\tfrac{4}{3}, \tfrac{17}{6} - \tfrac{\sqrt{465}}{18}\right\}\qquad\text{and}\qquad b_0\in \left\{1, \tfrac{23}{3} - \tfrac{\sqrt{465}}{3}\right\}.
\end{equation}
Assume initially that $a_0=2$.  Then $\Delta_a(2,b) = -2(b-1)(3b^2-34b+32)$, whose roots are $b\in\{1, \frac{17}{3} \pm \frac{\sqrt{193}}{3}\}$. By comparing them to the possibilities in \eqref{EQ:possibilities_for_a_b}, we find $b_0=1$.  But $(a_0,b_0)=(2,1)$ is not in the interior of $T$, so $a_0=2$ yields no interior critical points.

Next, assume that $a_0= \frac{4}{3}$.  We have $\Delta_b(\frac{4}{3},b) = \frac{2}{3} b (6b^2 -63b + 19)$, with roots $b\in \{0, \frac{21}{4} \pm \frac{\sqrt{3513}}{12}\}$. None of them match with the possibilities in \eqref{EQ:possibilities_for_a_b}, so $a_0=\frac{4}{3}$ also yields no interior critical points. It follows that any interior critical point $(a_0,b_0)$ must have $a_0 = \frac{17}{6} - \frac{\sqrt{465}}{18}$.

Now, we rule out  the possibility that $b_0=1$, and thus conclude that the only possible interior critical point is $(a_0,b_0) = \left(\frac{17}{6} - \frac{\sqrt{465}}{18}, \frac{23}{3}-\frac{\sqrt{465}}{3}\right).$  We have $\Delta_a(a,1) = 2(a-2)(18a^2-27a-17)$, with roots $a\in\{2, \frac{3}{4} \pm \frac{\sqrt{217}}{12}\}$.  Since neither of these is equal to $\frac{17}{6} - \frac{\sqrt{465}}{18}$, we deduce that $b_0\neq 1$.  
\end{proof}

\begin{remark}\label{rem:proc}

In the following lemma and elsewhere, we will need to understand  the zero set of an equation of the form $P\sqrt{Q} + R = 0$ where $P,Q,$ and $R$ are polynomial expressions in $a$ and $b$. We follow the standard squaring procedure: isolating the square root and squaring both sides to obtain $P^2 Q - R^2 = 0$. The expression $P^2 Q - R^2$ often factors, allowing us to analyze the zero set of each factor separately.

While this procedure is straightforward, when $P$, $Q$, and $R$ are complicated, the resulting formulas can get very complicated.  As such, we frequently employed the use of Maple to carry out the computations.  We encourage the interested reader to use their favorite Computer Algebra System.

Finally, we caution that moving from $P\sqrt{Q} = -R$ to $P^2Q = R^2$ has the potential to create extraneous solutions, i.e., solutions of the equation $P^2 Q - R^2 = 0$ which are not solutions of the original equation. The solutions to the equation $P\sqrt{Q} + R = 0$ are a subset of the solutions we find.  
\end{remark}

\bigskip

\begin{lemma}\label{LEM:A1nonZero}
There are no metrics of the form $g_{x,y,z,1}$ in $\mathcal B_P\setminus\mathcal (B_H\cup \mathcal{B}_V)$ with $(a,b)\in T\cap\{A\neq 0\}$ and $b-a+x>0$.
\end{lemma}

\begin{proof}

Suppose there exists such a metric. By Lemma~\ref{LEM:conditions_P1_0_case1}, the value of $x$ must be $x_1$ or $x_2$. We will see that both cases yield a contradiction. Recall throughout the proof that both $x_1$ and $x_2$ are continuous functions on $T\cap\{A\neq 0\}$, see Lemma~\ref{LEM:x1_x42_continuous_on_T-A}. Note that $T\cap\{A\neq 0\}$ has three connected components, namely
$$
U_1:=T\cap\{A> 0\},\qquad U_2:=T\cap\{A< 0 \text{ and } a < 2\}, \qquad U_3:=T\cap\{A<0 \text{ and } a > 2\},
$$
see Figures~\ref{fig:TA0} and \ref{fig:TA0zoom}. When dealing with $x_1$ and $x_2$, we will analyze their behavior in each $U_i$ separately.

\begin{figure}
\centering

\begin{minipage}{0.5\textwidth}
\centering
\begin{tikzpicture}[scale=1.15]
\begin{axis}[
    axis lines=middle, xlabel={$a$}, ylabel={$b$},
    xmin=-0.2, xmax=3.2, ymin=-0.2, ymax=2.2,
    xtick={0,1,2,3}, ytick={0,1,2},
    axis line style={->}, axis on top
]
    \addplot[name path=para, thick, domain=0:2.666, samples=100] {x^2/4};
    \path[name path=bot1] (axis cs:0,0) -- (axis cs:1.333,0);
    \path[name path=bot2] (axis cs:1.333,0) -- (axis cs:2.666, 1.777);
    \addplot[name path=line, dashed, thick, domain=0.8:2.6] {x-1};
    
    % Shading U2 (Lightest)
    \addplot[green!10] fill between[of=para and bot1, soft clip={domain=0:1}];
    \addplot[green!10] fill between[of=para and line, soft clip={domain=1:2}];
    
    % Shading U3 (Medium)
    \addplot[blue!40] fill between[of=para and line, soft clip={domain=2:2.333}];
    \addplot[blue!40] fill between[of=para and bot2, soft clip={domain=2.333:2.666}];
    
    % Shading U1 (Darkest)
    \addplot[red!70] fill between[of=line and bot1, soft clip={domain=1:1.333}];
    \addplot[red!70] fill between[of=line and bot2, soft clip={domain=1.32:2.333}];

    \draw[thick] (axis cs:0,0) -- (axis cs:1.333,0) -- (axis cs:2.666, 1.777);
    \node at (axis cs:0.95, .115) {$U_2$};
    \node at (axis cs:2.2, 1.45) {$U_3$};
    \node at (axis cs:1.28, 0.115) {$U_1$};

\end{axis}
\end{tikzpicture}
\caption{The region $T$ with $A=0$ (dashed) and the three corresponding regions} \label{fig:TA0}
\end{minipage}\hfill
\begin{minipage}{0.5\textwidth}
\centering
\begin{tikzpicture}
\begin{axis}[
    axis lines=middle, xlabel={$a$}, ylabel={$b$},
    xmin=1.5, xmax=2.8, ymin=0.5, ymax=2.0,
    xtick={1.5, 2.0, 2.5}, ytick={0.5, 1.0, 1.5, 2.0},
    axis line style={->}, axis on top,
    % This ensures the fill calculation doesn't get discarded
    restrict y to domain=0:2.5, 
    restrict x to domain=1:3
]
    % Define boundaries - using slightly wider domains than the xmin/xmax
    \addplot[name path=paraZ, thick, domain=1.4:2.666, samples=100] {x^2/4};
    % Define the slanted bottom line explicitly as a plot so fillbetween handles it better
    \addplot[name path=botZ, thick, domain=1.333:2.666] {1.3333*(x-1.3333)};
    \addplot[name path=lineZ, dashed, thick, domain=1.4:2.7] {x-1};
    
    % Shading U2 (Lightest) - Domain limited by zoom xmin
    \addplot[green!10] fill between[of=paraZ and lineZ, soft clip={domain=1.5:2}];
    
    % Shading U3 (Medium)
    \addplot[blue!40] fill between[of=paraZ and lineZ, soft clip={domain=2:2.333}];
    \addplot[blue!40] fill between[of=paraZ and botZ, soft clip={domain=2.333:2.666}];
    
    % Shading U1 (Darkest) - Start domain exactly at zoom xmin (1.5)
    \addplot[red!70] fill between[of=lineZ and botZ, soft clip={domain=1.5:2.333}];

    % Redraw visible boundary line to be crisp
    \draw[thick] (axis cs:1.333,0) -- (axis cs:2.666, 1.777);

    \node at (axis cs:1.6, 0.75) {$U_2$};
    \node at (axis cs:2.3, 1.45) {$U_3$};
    \node at (axis cs:1.72, 0.62) {$U_1$};

\end{axis}
\end{tikzpicture}\caption{A zoomed-in view of the top right portion of Figure \ref{fig:TA0}}\label{fig:TA0zoom}
\end{minipage}\hfill
\end{figure}

In order to rule out $x_1$, we will show that, under the given assumptions, $x_1\geq\frac{4}{3}$ on $U_1$ and $x_1\leq 0$ on both $U_2$ and $U_3$. Note that $x_1\geq\frac{4}{3}$ is equivalent to $H_1\leq 0$, which is a contradiction because any metric in $\mathcal B_P\setminus(\mathcal B_H \cup \mathcal{B}_V)$ must have $H_1>0$. Moreover, $x_1\leq 0$ means that $g_{x_1,y,z,1}$ does not even represent a metric, so is obviously a contradiction.

We begin by showing that $x_1\leq 0$ on $U_2$ and $U_3$. Because  $A<0$ on $U_2$ and $U_3$, we only need to prove that $-B+\sqrt{\Delta}\geq 0$. Since $C=a^2-4b\geq 0$ on $T$, we have 
$$
\sqrt{\Delta}=\sqrt{B^2-4AC}\geq \vert B\vert\geq B
$$
on both $U_2$ and $U_3$. This implies that $-B+\sqrt{\Delta}\geq 0$, as desired.

Next, we show that $x_1 \geq \frac{4}{3}$ on $U_1$.  Consider the equation $x_1=\frac{4}{3}$. Since $A\neq 0$, this is equivalent to $-3B+3\sqrt{\Delta}=8A$. Following the approach in Remark~\ref{rem:proc}, we arrive at the equation $9B^2-36AC=64A^2+9B^2+48AB$. Dividing by $A\neq 0$ we get $64A+48B+36C=0$, which can be seen to be equivalent to $ \left(3 a-2 b-4\right)^{2}=0$. Thus, $x_1=\frac{4}{3}$ can occur
only along the line $3 a-2 b-4=0$. This line does not intersect $U_2$ or $U_3$ and it divides $U_1$ into two connected components (see Figure~\ref{fig:feasible1_A_pos}, where the dotted line depicts $-2b+3a-4=0$). Altogether we conclude that, on each subregion of $U_1$, the value of $x_1$ is either larger than $\frac{4}{3}$ or smaller than $\frac{4}{3}$. We take a sample point in each subregion:
$$
\left(\frac{3}{2}, \frac{1}{3}\right)\in U_1\cap \left\{{ b>\frac{3}{2}}a-2\right\},\qquad \left(2, \frac{32}{35}\right)\in U_1\cap \left\{{ b<\frac{3}{2}}a-2\right\}.
$$
We compute that $x_1\left(\frac{3}{2},\frac{1}{3}\right) = \frac{3}{2}  > \frac{4}{3}$ and $x_1\left(2,\frac{32}{35}\right) = \frac{10}{7} > \frac{4}{3}$, so $x_1\geq\frac{4}{3}$ on $U_1$, as claimed.

\begin{figure}
\centering

\begin{minipage}{0.5\textwidth}
\centering

\begin{tikzpicture}[scale=1.15]
\begin{axis}[
    axis lines=middle, xlabel={$a$}, ylabel={$b$},
    xmin=-0.2, xmax=3.2, ymin=-0.2, ymax=2.2,
    xtick={0,1,2,3}, ytick={0,1,2},
    axis line style={->}, axis on top
]
    \addplot[name path=para, thick, domain=0:2.666, samples=100] {x^2/4};
    \path[name path=bot1] (axis cs:0,0) -- (axis cs:1.333,0);
    \path[name path=bot2] (axis cs:1.333,0) -- (axis cs:2.666, 1.777);
    \addplot[name path=line, dashed, thick, domain=0.8:2.6] {x-1};
    \addplot[dotted, thick, domain=0:2.5]{0.5*(3*x-4)};
    
    % Shading U2 (Lightest)
    \addplot[green!10] fill between[of=para and bot1, soft clip={domain=0:1}];
    \addplot[green!10] fill between[of=para and line, soft clip={domain=1:2}];
    
    % Shading U3 (Medium)
    \addplot[blue!40] fill between[of=para and line, soft clip={domain=2:2.333}];
    \addplot[blue!40] fill between[of=para and bot2, soft clip={domain=2.333:2.666}];
    
    % Shading U1 (Darkest)
    \addplot[red!70] fill between[of=line and bot1, soft clip={domain=1:1.333}];
    \addplot[red!70] fill between[of=line and bot2, soft clip={domain=1.32:2.333}];

    \draw[thick] (axis cs:0,0) -- (axis cs:1.333,0) -- (axis cs:2.666, 1.777);
    \node at (axis cs:0.95, .115) {$U_2$};
    \node at (axis cs:2.2, 1.45) {$U_3$};
    \node at (axis cs:1.28, 0.115) {$U_1$};
    
\end{axis}
\end{tikzpicture}
\caption{The region $T$ with $A=0$ (dashed) and $-2b+3a-4 = 0$ (dotted)} 
\label{fig:feasible1_A_pos}

\end{minipage}\hfill
\begin{minipage}{0.5\textwidth}
\centering
\begin{tikzpicture}[scale=1.15]
\begin{axis}[
    axis lines=middle, xlabel={$a$}, ylabel={$b$},
    xmin=-0.2, xmax=3.2, ymin=-0.2, ymax=2.2,
    xtick={0,1,2,3}, ytick={0,1,2},
    axis line style={->}, axis on top
]
    \addplot[name path=para, thick, domain=0:2.666, samples=100] {x^2/4};
    \path[name path=bot1] (axis cs:0,0) -- (axis cs:1.333,0);
    \path[name path=bot2] (axis cs:1.333,0) -- (axis cs:2.666, 1.777);
    \addplot[name path=line, dashed, thick, domain=0.8:2.6] {x-1};
    \addplot[dotted, thick, domain=0:2.5]{2*x-3};
    
    % Shading U2 (Lightest)
    \addplot[green!10] fill between[of=para and bot1, soft clip={domain=0:1}];
    \addplot[green!10] fill between[of=para and line, soft clip={domain=1:2}];
    
    % Shading U3 (Medium)
    \addplot[blue!40] fill between[of=para and line, soft clip={domain=2:2.333}];
    \addplot[blue!40] fill between[of=para and bot2, soft clip={domain=2.333:2.666}];
    
    % Shading U1 (Darkest)
    \addplot[red!70] fill between[of=line and bot1, soft clip={domain=1:1.333}];
    \addplot[red!70] fill between[of=line and bot2, soft clip={domain=1.32:2.333}];

    \draw[thick] (axis cs:0,0) -- (axis cs:1.333,0) -- (axis cs:2.666, 1.777);
    \node at (axis cs:0.95, .115) {$U_2$};
    \node at (axis cs:2.2, 1.45) {$U_3$};
    \node at (axis cs:1.28, 0.115) {$U_1$};

\end{axis}
\end{tikzpicture}\caption{The region $T$ with $A=0$ (dashed) and $-b-3+2a=0$ (dotted)}\label{fig:feasible2_A_pos}
\end{minipage}\hfill
\end{figure}

Now we focus on $x_2$. We will show that $x_2$ satisfies $2b+3x_2-3a \leq 0$ on $T\cap\{A\neq 0\}$, so Lemma~\ref{LEM:conditions_P1_0_case1} implies that $x_2$ does not yield a metric in $\mathcal B_P\setminus(\mathcal B_H\cup\mathcal B_V)$.

To prove that $2b+3x_2-3a{<}0$ on $T\cap\{ A\neq 0\}$, we will see that the graph of $2b+3x_2-3a=0$ does not intersect $T\cap\{ A\neq 0\}$. Assuming this, at the sample points
$$
\left(\frac{3}{2}, \frac{1}{3}\right)\in U_1, \qquad \left(\frac{2}{3}, \frac{1}{19}\right)\in U_2,\qquad \left(\frac{43}{18}, \frac{17}{12}\right)\in U_3,
$$
we compute that $2b + 3x_2 - 3a$ is given by respectively $-\frac{1}{6}, -\frac{35}{209}, -\frac{13}{12},$ so all three are negative. Hence $2b+3x_2-3a<0$ on $T\cap\{A\neq 0\}$.

Consider the equation $2b+3x_2-3a=0$. Substitute $x_2=\frac{-B-\sqrt{\Delta}}{2A}$ and multiply both sides by $\frac{2A}{3}$. This gives $\frac{4}{3}bA-B-\sqrt{\Delta}-2A = 0$, which can be rearranged as $\left(\frac{4}{3}b-2\right)A-B= \sqrt{\Delta}$. Following the approach in Remark~\ref{rem:proc}, we square both sides and work out the expression, and we arrive at the equation
$$
0 =  A b \left(-2 b+3 a-4\right) \left(-b-3+2 a\right).
$$
Since $Ab\neq 0$ on $T\cap\{A\neq 0\}$, it follows that $2b+3x_2-3a$ can only vanish if $-2 b+3 a-4=0$ or $-b-3+2 a=0$. Let us show that neither of these possibilities occur.

Suppose $-2 b+3 a-4=0$. As mentioned in the case of $x_1$ and illustrated in Figure~\ref{fig:feasible1_A_pos}, this line does not intersect $U_2$ and $U_3$ and divides $U_1$ into two connected components. Moreover, the intersection of $-2 b+3 a-4=0$ and $U_1$ occurs only for $a\in \left(\frac{4}{3},2\right]$. Substituting $b=\frac{3}{2}a-2$, we get $A=-\frac{3}{2}(a-2) $, $B=-\frac{1}{4}(a-2)(3a-20)$, $C=(a-2)(a-4)$, and $\Delta=\frac{1}{16}(a-2)^2(3a-4)^2$. Since $a\in \left(\frac{4}{3},2\right]$ along the intersection, we conclude that $\sqrt{\Delta} = - \frac{1}{4}(a-2)(3a-4)$, and hence $x_2=2-\frac{a}{2}$. Consequently,  $2b+3x_2-3a=-\frac{3}{2}a+2 $ equals $-b$, and is therefore negative on all of $T$. 

Suppose $-b-3+2 a=0$. This line also does not intersect $U_2$ and $U_3$ and divides $U_1$ into two connected regions, see Figure \ref{fig:feasible2_A_pos}. This intersection of $-b-3+2 a=0$ and $U_1$ occurs for $a\in\left( \frac{11}{6} ,2 \right)$. Substituting $b=-3+2a$, we obtain $A=-3(a-2)$, $B=-(a-2)(a-9)$, $C=(a-2)(a-6)$, and $\Delta=(a-2)^2(a-3)^2$. For $a\in (0,2)$ we find that $x_2=1$. Hence $2b+3x_2-3a=a-3$, which is negative for $a\in\left( \frac{11}{6} ,2 \right)$. This finishes the proof.
\end{proof}

\subsection{\texorpdfstring{The case where $b-a+x<0$}{The case where b-a+x>0}}\label{SUBSEC:bax_neg}

The main results of this section are Proposition~\ref{prop:possible_x4} and Corollary~\ref{COR:BP_minus_BH}, which ultimately yield the set containing $\mathcal{B}_P\setminus(\mathcal{B}_H\cup\mathcal{B}_V)$.  Before stating them, we need to introduce some notation.

Because $b-a+x<0$, the equation $P_1=0$ is equivalent to $\sqrt{H_1V_1}=-4b+3a-3x$. In order to have a solution with both $H_1$ and $V_1$ non-zero, we must have
\begin{equation}\label{eq:ineq_case_bax_negative}
    -4b+3a-3  >  0.
\end{equation}
Under the assumption in~\eqref{eq:ineq_case_bax_negative}, $P_1=0$ is equivalent to $H_1V_1=(-4b+3a-3x)^2$. Expanding and rearranging this equation, we find that it is equivalent to:
\begin{equation}\label{EQ:Px34} 
3\left(a-2 b-1\right) x^{2}+\left(-3 a^{2}+2 \left(1+3 b\right) a-4 b^{2}+3 b\right) x+a^{2}-4 b=0.
\end{equation}
The left-hand side of \eqref{EQ:Px34} is a polynomial $Dx^2+Ex+C$, where:
\begin{align*}
    D &:= 3 \left(a-2 b-1\right), \\
   E &:= -3a^{2}+2 \left(1+3 b\right) a-4 b^{2}+3 b,  \\
     C &:= a^{2}-4b.
\end{align*}
The observations above prove the following lemma.

\begin{lemma}\label{LEM:conditions_P1_0_case2}
    A metric $g_{x,y,z,1}$ with $b-a+x<0$ belongs to $\mathcal{B}_P\setminus(\mathcal{B}_H\cup \mathcal{B}_V)$ if and only if the following two conditions are satisfied:
    \begin{itemize}
        \item $-4b+3a-3x > 0$, and 
        \item $Dx^2+Ex+C=0$.
    \end{itemize}
\end{lemma}

In order to analyze the solutions of $Dx^2+Ex+C=0$, we divide our arguments into two cases: $D=0$ and $D\neq 0$. We first rule out the case $D= 0$.

\begin{lemma}\label{LEM:A2zero}
There are no metrics of the form $g_{x,y,z,1}$ in $\mathcal B_P\setminus(\mathcal B_H \cup \mathcal{B}_V)$ with $(a,b)\in T\cap\{D=0\}$ and $b-a+x<0$.
\end{lemma}

\begin{proof}
Suppose for a contradiction that there exists such a metric. By Lemma~\ref{LEM:conditions_P1_0_case2}, the parameter $x$ must satisfy $Ex+C=0$. We will show that any solution to $Ex+C=0$ satisfies $-4b+3a-3x<0$, hence giving the desired contradiction.

As a first step we study the set $T\cap\{D=0\}$. Note that $D=0$ is equivalent to $b = \frac{1}{2}(a-1)$. Since $b>0$ in $T$, the points in the intersection must have $a > 1$.  Moreover, the line $b = \frac{1}{2}(a-1)$ intersects the line $b =\frac{4}{3}\left(a-\frac{4}{3}\right)$ (which is a bounding curve of $T$) when $a = \frac{23}{15}$. Thus, we find $a \in (1,\frac{23}{15}]$.

Now consider the equation $Ex+C=0$. We observe that in this case $E$ cannot vanish; indeed, if we substitute $b = \frac{1}{2}(a-1)$ in $E$ we obtain $E=-\frac{1}{2}(2a^2-5a+5)$, which is a polynomial with no real roots. It follows that $x=-\frac{C}{E}$. Substituting $b = \frac{1}{2}(a-1)$ into $C$ gives $C=a^2-2a+2$, hence
$$
x = \frac{2(a^2-2a+2)}{2a^2-5a+5}.
$$

Finally, substituting $b = \frac{1}{2}(a-1)$ and $x = \frac{2(a^2-2a+2)}{2a^2-5a+5}$ in $-4b+3a-3x$ gives:
$$
-4b+3a-3x=-\frac{(1-a)(a-2)(2a-1)}{2a^2-5a+5}.
$$
The denominator equals $-2E$ which, as we have observed above, is never zero. Moreover, $2a^2-5a+5$ is positive for all $a\in\R$. As for the numerator, it is straightforward to check that it is positive for all $a\in \left(1,\frac{23}{15}\right]$. The minus sign in front of the fraction implies that $-4b+3a-3x<0$, as claimed at the beginning of the proof.
\end{proof}

Now we consider the case $D\neq 0$. The solutions to \eqref{EQ:Px34} are the roots of a quadratic polynomial in the variable $x$:
\begin{equation}\label{EQ:x3_and_x4}
    x_3=\frac{-E+\sqrt{\Gamma}}{2D},\qquad x_4=\frac{-E-\sqrt{\Gamma}}{2D},
\end{equation}
where in this case $\Gamma := E^2 - 4DC$. We make a couple of observations about $\Gamma$ in Lemmas~\ref{LEM:SimplifyGamma} and \ref{LEM:x3_x4_continuous_on_T-A}.

\begin{lemma}\label{LEM:SimplifyGamma} 
Along the curve $b = \frac{a^2}{4}$ with $a\in(0,\frac{8}{3}]$, we have $\Gamma = -\frac{1}{4} a (a^3-6a^2+9a-8)\neq 0$.
\end{lemma}

\begin{proof} It is straight forward to verify that $$\Gamma = \sqrt{ \frac{1}{16}a^2 (a^3-6a^2+9a-8)} = \frac{1}{4}| a (a^3-6a^2+9a-8)|.$$   The maximum value of the cubic factor on $[0,\frac{8}{3}]$ is $-4$, occurring when $a=1$.  This follows since its derivative is $3(a-1)(a-3)$.  Thus, as $a\in(0,\frac{8}{3}]$, it follows that $|a (a^3-6a^2 +9a-8| = - a(a^3-6a^2+9a-8)\neq 0$, completing the proof.
\end{proof}

\begin{lemma}\label{LEM:x3_x4_continuous_on_T-A}
For any $(a,b)\in T$ it holds that $\Gamma \geq 0$.  In particular, the formulas for $x_3$ and $x_4$ define continuous functions on $T\setminus \{D=0\}$.
\end{lemma}

\begin{proof}
In order to prove that $\Gamma\geq 0$ on $T$, we will show that the zero set of $\Gamma$ intersects the closure of $T$ only in isolated points.  Since $\Gamma$ is a continuous function of $(a,b)$, it will then follow that on $T$ either $\Gamma \geq 0$ or $\Gamma \leq 0$. Since $\Gamma\left(1,\frac{1}{8}\right) = \frac{385}{256} > 0$, we conclude that $\Gamma \geq 0$ on $T$.

We begin by showing that the zero set of $\Gamma$ intersects the boundary of $T$ only at the points $(0,0)$ and $\left(\frac{4}{3},0\right)$.  First consider the boundary curve $b = \frac{a^2}{4}$ with $a\in \left[0,\frac{8}{3}\right]$. When $a=0$ we get the obvious solution $(0,0)$. For $a\in(0,\frac{8}{3}]$, Lemma~\ref{LEM:x3_x4_continuous_on_T-A} implies that there are no further points in the zero set. Second, along the boundary curve $b=0$ with $a\in \left(0,\frac{4}{3}\right]$ the equation $\Gamma=0$ simplifies to $a^2(3a-4)^2 = 0$, which yields the point $\left(\frac{4}{3},0\right)$. Finally, along the curve $b = \frac{4}{3}\left(a-\frac{4}{3}\right)$ the equation $\Gamma = 0$ is equivalent to
$$
(3a-4)^2 ( 3249a^2 - 20760a + 43072)=0.
$$  
The quadratic factor has a discriminant of $-128786112 < 0$, so it has no real zeros. Since the boundary curve $b = \frac{4}{3}\left(a - \frac{4}{3}\right)$ has domain $a\in \left(\frac{4}{3},\frac{8}{3}\right]$, we conclude that there are no additional intersections in this case.

Now we analyze the behavior of $\Gamma$ in small neighborhoods of the points $(0,0)$ and $\left(\frac{4}{3},0\right)$. Consider initially the point $(0,0)$. We claim that the zero set of $\Gamma$ is disjoint from $T$ in a deleted neighborhood of $(0,0)$.  Since the partial derivative $\Gamma_b(0,0)\neq 0$, by the implicit function theorem,  there is a neighborhood of $a =0\in \mathbb{R}$ for which the set $\Gamma = 0$ equals the graph of some function $b = g(a)$.  We claim that $g(a) > \frac{a^2}{4}$ for all $a$ in a possibly smaller deleted neighborhood of $a=0$, so that $(a,g(a))\notin T$.  To verify this, we compute that the first derivative of $g$ at $a=0$ is equal to $-\frac{\Gamma_a}{\Gamma_b}=0$,  and the second derivative is $\frac{2}{3}$. For the curve $b = \frac{a^2}{4}$, the first derivative at $a=0$ is $0$ and the second derivative is $\frac{1}{2}$.  Since $\frac{2}{3} > \frac{1}{2}$, we conclude that $g(a) > \frac{a^2}{4}$ for small $a>0$. It follows that in a deleted neighborhood of $(0,0)$, the zero set of $\Gamma$ is disjoint from $T$.

We next claim that $\left(\frac{4}{3},0\right)$ is an isolated point of the zero set of $\Gamma$.  We observe that $\left(\frac{4}{3},0\right)$ is a critical point of $\Gamma$ in the sense that both partial derivatives, $\Gamma_a$ and $\Gamma_b$, vanish at $\left(\frac{4}{3},0\right)$. In addition, the Hessian $\left(\begin{smallmatrix}\Gamma_{aa} & \Gamma_{ba}\\\Gamma_{ab}&\Gamma_{bb} \end{smallmatrix}\right)$ of $\Gamma$ at $\left(\frac{4}{3},0\right)$ is $\left(\begin{smallmatrix} 32 & -52 \\ -52 & \frac{278}{3}\end{smallmatrix}\right)$, which is positive definite.  It follows that $\left(\frac{4}{3},0\right)$ is a strict local minimum of $\Gamma$, so that $\left(\frac{4}{3},0\right)$ is an isolated point of the zero set.

To conclude the proof, we will show that the zero set of $\Gamma$ does not intersect the interior of $T$.  Our first step is to show that the intersection would necessarily be a smooth $1$-manifold of the interior of $T$.  From the implicit function theorem, it is sufficient to verify that at least one of the partial derivatives of $\Gamma$ is non-zero at every point in the interior of $T$.

Towards that end, we first observe that $C\geq 0$ for all $(a,b)\in T$, and moreover $C>0$ if $b\neq \frac{a^2}{4}$.  Thus, in order to have $E^2 - 4DC = 0$, it is necessary that $D > 0$, or equivalently $b < \frac{a-1}{2}$.  We note that $(a,b)\in T$ and $D > 0$ implies $a\in \left(1,\frac{23}{15}\right]$ and $b\in \left(0,\frac{4}{15}\right]$.

Now, in order to find common zeros of $\Gamma_a$ and $\Gamma_b$, we compute the resultant $\operatorname{res}_b(\Gamma_a,\Gamma_b)$, see Section~\ref{SEC:resultant_Sturm} for details:
$$
\operatorname{res}_b(\Gamma_a,\Gamma_b) = 
-1179648 (3a - 4)(a + 1)(216a^5 - 1116a^4 + 2240a^3 - 2220a^2 + 933a - 58).
$$
Applying Sturm's Theorem (see also Remark~\ref{REM:Sturm}) to the last factor, we see that it has no zeros with $a\in \left(1,\frac{23}{15}\right]$.  The factor $(a+1)$ obviously has no zeros with $a\in \left(1,\frac{23}{15}\right]$.  Thus, we conclude that any common zero $(a_0,b_0)$ of $\Gamma_a$ and $\Gamma_b$ must have $a_0 = \frac{4}{3}$.  Substituting $a = \frac{4}{3}$ into the equation $\Gamma = 0$, we find $\frac{1}{3} b^2 ( 48b^2 -264b + 139) = 0$, so $b_0$ must be either equal to $0$ or a root of the quadratic factor.  The case $b_0=0$ corresponds to a boundary point of $T$, while the roots of the quadratic factor are $\frac{11}{4} \pm \frac{\sqrt{42}}{3}$, both of which are larger than $\frac{4}{15}$.

It follows that the intersection of the zero set of $\Gamma$ and the interior of $T$ is a relatively closed embedded submanifold.  Since we have already shown that the intersection of the zero set of $\Gamma$ and the boundary consists of isolated points, it follows that, in fact, the intersection of the zero set of $\Gamma$ and the interior of $T$ can contain no components homeomorphic to an interval.  In other words, the intersection must be, up to diffeomorphism, a disjoint union of a finite collection of circles.  We conclude the proof by showing there are no such circles.

To see this, assume for a contradiction that there is some circle. Then there are at least two points (the top most and the bottom most) where the tangent lines are horizontal, so that $\Gamma_a = 0$. To find common zeros $(a_0,b_0)$ of $\Gamma$ and $\Gamma_a$ we compute the resultant:
$$
\operatorname{res}_b(\Gamma,\Gamma_a) = 1179648 a(72a^5 - 588a^4 + 1772a^3 - 2940a ^2 + 2631a - 1096)(a + 1)^2(3a - 4)^2.
$$  
Since $a_0\in \left(1,\frac{23}{15}\right]$, we conclude that either $a_0 = \frac{4}{3}$ or $a_0$ is a root of $(72a^5 - 588a^4 + 1772a^3 - 2940a^2 + 2631a - 1096)$.  In the first case, as argued above, we find $b_0 = 0$, obtaining a boundary point instead of an interior point. In the second case, Sturm's Theorem shows it has no zeros in $\left(1,\frac{23}{15}\right]$, completing the proof.
\end{proof}

\begin{lemma}\label{LEM:no_x3}  
There are no metrics of the form $g_{x,y,z,1}$ in $\mathcal B_P\setminus(\mathcal B_H \cup \mathcal{B}_V)$ with $(a,b)\in T\cap\{D\neq 0\}$, with $b-a+x<0$ and with $x=x_3(a,b)$.
\end{lemma}

\begin{proof} 
Suppose there exists such a metric. Recall from Lemma~\ref{LEM:x3_x4_continuous_on_T-A} that $x_3$ defines a continuous function on $T\cap \{D\neq 0\}$. Note that $T\cap\{D\neq 0\}$ has two connected components, namely $T\cap\{D> 0\}$ and $T\cap\{D< 0\}$, see Figure~\ref{fig:TD0}. We will denote
$$
W_1:=T\cap\{D> 0\},\qquad W_2:=T\cap\{D< 0\}.
$$
We will show that, under the given assumptions, $x_3>\frac{4}{3}$ on $W_1$ and $x_3\leq 0$ on $W_2$. This concludes the proof, since any metric $g_{x_3,y,z,1}\in\mathcal B_P\setminus(\mathcal B_H \cup \mathcal{B}_V)$ must have $x_3\in (0,\frac{4}{3})$.

We first show that $x_3\leq 0$ on $W_2$. This is equivalent to $-E+\sqrt{\Gamma}\geq 0$ because $D<0$ on $W_2$. Since $C=a^2-4b\geq 0$ on $T$, we have 
$$
\sqrt{\Gamma}=\sqrt{E^2-4DC}\geq \vert E\vert\geq E
$$
on $W_2$. Thus we get that $-E+\sqrt{\Gamma}\geq 0$, as desired.

Now we show that $x_3>\frac{4}{3}$ on $W_1$. A simple calculation shows that $x_3 = \frac{4}{3}$ if and only if $4D(-4+3a-4b)^2 = 0$. As $D\neq 0$, we find that $x_3 = \frac{4}{3}$ if and only if $b = -1 + \frac{3}{4}a$. We claim that the graph of $b = -1 + \frac{3}{4} a$ does not intersect the region $T$, and in particular it does not intersect $W_1$. To prove this, consider a point $(a,b)$ with $b = -1 + \frac{3}{4} a$. If $a\leq \frac{4}{3}$, we find $b\leq 0$, so $(a,b)\notin T$. If otherwise $a>\frac{4}{3}$, then we observe that $b<\frac{4}{3}a-\frac{16}{9}$, so $(a,b)\notin T$ either.

Since $x_3$ is continuous on $W_1$, it follows that either $x_3<\frac{4}{3}$ or $x_3>\frac{4}{3}$ for all $(a,b)\in W_1$. Taking the sample point $(a,b) = (\frac{7}{5}, \frac{1}{10})\in W_1$, we see that $x_3\left(\frac{7}{5}, \frac{1}{10}\right) = 2> \frac{4}{3}.$
Consequently $x_3 > \frac{4}{3}$ on $W_1$, as claimed. 
\end{proof}

\begin{figure}
\begin{center}
\begin{tikzpicture}
\begin{axis}[
    axis lines=middle,
    xlabel={$a$},
    ylabel={$b$},
    xmin=-0.2, xmax=3.2,
    ymin=-0.2, ymax=2.2,
    xtick={0,1,2,3},
    ytick={0,1,2},
    axis line style={->},
    axis on top,
    set layers,
]
    % Define the upper boundary (parabola)
    \addplot[name path=para, thick, domain=0:2.666, samples=100] {x^2/4};
    \addplot[name path=line, dashed, thick, domain=0.5:2.8, samples=100]{0.5*(x-1)};
    
    % Define the piecewise lower boundary
    % Path goes from (0,0) to (4/3,0) to (8/3, 16/9)
    \path[name path=bot] (axis cs:0,0) -- (axis cs:1.333,0) -- (axis cs:2.666, 1.777);
    
    % Fill the region R in grayscale
    \addplot[yellow!20] fill between[of=para and bot];
    
    % Draw the lower boundary line explicitly
    \draw[thick] (axis cs:0,0) -- (axis cs:1.333,0) -- (axis cs:2.666, 1.777);

    % Optional: Mark the tip of the region
    \filldraw (axis cs:2.666, 1.777) circle (1pt);

    %Label regions W1 and W2
    \node[below, label={$W_1$}] at (1.6, 0.005) {};
    \node[below, label={$W_2$}] at (1,0.005) {};

\addplot[teal!70] fill between[of=line and bot, soft clip={domain=1:1.533}];

\end{axis}
\end{tikzpicture}
\caption{The region $T$ with the line $D=0$}\label{fig:TD0}
\end{center}
\end{figure}
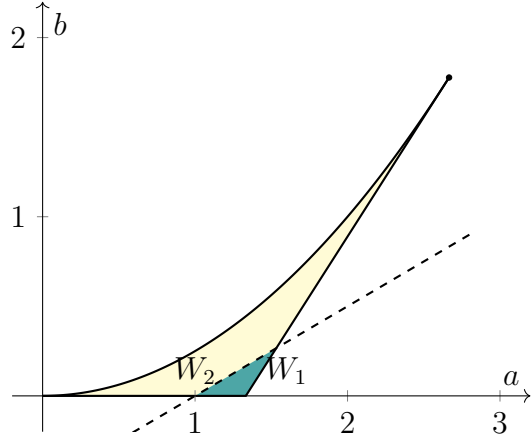

At this point, it follows from Lemmas~\ref{LEM:absvaluenonzero}, \ref{LEM:no_bax_positive}, \ref{LEM:A2zero} and \ref{LEM:no_x3} that any metric in $\mathcal B_P\setminus(\mathcal B_H\cup \mathcal{B}_V)$ must be of the form $g_{x,y,z,1}$ with $(a,b)\in T$, $b-a+x<0$ and $x=x_4(a,b)$. Our next result shows that in fact $g_{x_4,y,z,1}\in \mathcal B_P\setminus(\mathcal B_H\cup\mathcal{B}_V)$ implies that $(a,b)$ must be in the subset $S\subset T$ defined as:
\begin{equation}\label{EQ:set_S}
    S:=\left\{ (a,b)\in\R^2\; : \; 0<b\leq \frac{a^2}{4} \text{ and } b> a-\frac{3}{4}\right\}.
\end{equation}
More precisely, we have:

\begin{proposition}\label{prop:possible_x4}
A metric $g_{x,y,z,1}$ satisfies $b-a+x < 0$, $P_1=0$, $H_1 > 0$, $H_2, H_3 \geq 0$ and $V_1 > 0$ if and only if $(a,b)\in S$ and $x=x_4(a,b)$.
\end{proposition}

\begin{corollary}\label{COR:BP_minus_BH}
We have an inclusion
$$
\mathcal B_P\setminus (\mathcal B_H\cup \mathcal B_V)\subset\{ g_{x,y,z,1}\; :\; (a,b)\in S\text{ and } x=x_4(a,b)   \}.
$$
\end{corollary}

\begin{proof}[Proof of Proposition~\ref{prop:possible_x4}]
Let $g_{x,y,z,1}$ be a metric satisfying the hypotheses in the statement. Since $H_2$ and $H_3$ are non-negative, we know from Remark~\ref{REM:square_0_43} that $(y,z)\in(0,\frac{4}{3}]^2$ and hence $(a,b)\in T$. Since $P_1=0$, Lemmas~\ref{LEM:A2zero} and \ref{LEM:no_x3} imply that $D\neq 0$ and $x=x_4$. Recall from the proof of Lemma~\ref{LEM:no_x3} and Figure~\ref{fig:TD0} that $T\cap \{D\neq 0\}$ consists of two connected components $W_1$ and $W_2$, with $W_2$ containing the graph of $b= \frac{a^2}{4}$. We shall show that $-4b+3a-3x_4<0$ on $W_1$, which is a contradiction to Lemma~\ref{LEM:conditions_P1_0_case2}.

Indeed, recall from Lemma~\ref{LEM:x3_x4_continuous_on_T-A} that $x_4(a,b)$ defines a continuous function on $T\cap \{D\neq 0\}$. Consider the equation $-4b+3a-3x_4 = 0$. Substituting $x_4$ by its value and solving it for $b$, we find $b = 0$, $b=\frac{3}{4}a-1$ and $b =a-\frac{3}{4} $ are the only solutions. The graph of $b=0$ is not part of $T$, while $b = \frac{3}{4}a-1$ intersects the closure of $T$ only at $(\frac{4}{3},0)\notin T$.  The line $b=a-\frac{3}{4}$ intersects the interior of $T$. However, it does not intersect $W_1$, thus the sign of $-4b+3a-3x_4$ is constant on $W_1$. Using the sample point $(a,b)=\left(\frac{7}{5},\frac{1}{10}\right)\in W_1$ we find that $-4b+3a-3x_4 = - \frac{1}{10} <0$ on $W_1$.

It follows that $(a,b)$ must lie in $W_2$. Note that the line $b=a-\frac{3}{4}$ divides $W_2$ in two components and that one of the components is precisely the subset $S$, see Figure \ref{fig:W2S}. Thus, the sign of $-4b+3a-3x_4$ on $S$ is constant and it is also constant on $W_2\setminus S$. We take the sample points $(a,b)=\left(\frac{4}{5}, \frac{1}{7}\right)\in S$ and $(a,b)=\left(\frac{4}{3},\frac{1}{3}\right)\in W_2\setminus S$. For $\left(\frac{4}{5}, \frac{1}{7}\right)$ we find $-4b+3a-3x_4 =  \frac{4}{35}>0$.  Similarly, for $\left(\frac{4}{3}, \frac{1}{3}\right)$ we find $-4b+3a-3x_4 = -\frac{1}{3} < 0$. Thus, $-4b+3a-3x_4 \leq 0$ on $W_2\setminus S$. Altogether, we have shown that $-4b+3a-3x_4 > 0$ implies that $(a,b)\in S$.

Now we prove the converse. Let $g_{x,y,z,1}$ be a metric with $(a,b)\in S$ and $x=x_4$. Since $S\subset T$, we know that $H_2,H_3\geq 0$. Also, from the previous paragraph it is clear that $-4b+3a-3x_4{\color{teal}>}{\color{gray}\geq} 0$ on $S$. It remains to prove that $b-a+x_4<0$, $H_1(x_4,a,b)> 0$, and $V_1(x_4,a,b)> 0$ on $S$. 

To show that $b-a+x_4<0$, note that $b-a+x_4 = 0$ if and only if $b=0$, $b= \frac{3}{2}a-1$, or $b=a-1$.  None of these lines intersect $S$, so the sign of $b-a+x_4$ is constant on $S$. At $(a,b) = \left(\frac{4}{5}, \frac{1}{7}\right)$, we find $b-a+x_4 = - \frac{3}{35} < 0$, so $b-a+x_4<0$ on all of $S$.

We similarly find that $H_1(x_4,a,b) = 4-3x_4 = 0$ if and only if $b = \frac{3}{4}a-1$, whose graph does not intersect $S$. At $\left(\frac{4}{5}, \frac{1}{7}\right)$, we find $4-3x_4 = \frac{16}{7} > 0$, so we conclude that $H_1(x_4,a,b)>0$ for all $(a,b)\in S$.

Finally, following the procedure of Remark \ref{rem:proc}, we find that $V_1(x_4,a,b) = 0$ implies $$b (-4b-3+4a) (a-2b-1)(a^2-4b) = 0.$$   The first three factors yield curves which do not intersect $S$.  Finally, we argue $V_1(x_4,a,b)\neq 0$ when $b = \frac{a^2}{4}$.  Substituting $b=\frac{a^2}{4}$ in $D$ and $E$ and using the value of $\Gamma$ given in Lemma~\ref{LEM:SimplifyGamma}, we find that $x_4= \frac{-E-\sqrt{\Gamma}}{2D}= -\frac{a(a^3 - 6a^2 + 9a - 8)}{6a^2 - 12a + 12}$, from which it easily follows that $V_1 = \frac{a^2 (a-1)^2}{2(a^2-2a+2)}$.  Since the numerator is positive on $(0,1)$ and the denominator is always positive, it follows that $V_1 > 0$ along the boundary curve $b = \frac{a^2}{4}$.  It now follows that $V_1 > 0$ on all of $S$.
\end{proof}

\begin{figure}
\centering

\begin{minipage}{0.5\textwidth}
\centering
\begin{tikzpicture}
\begin{axis}[
    axis lines=middle,
    xlabel={$a$},
    ylabel={$b$},
    xmin=-0.2, xmax=3.2,
    ymin=-0.2, ymax=2.2,
    xtick={0,1,2,3},
    ytick={0,1,2},
    axis line style={->},
    axis on top,
    set layers,
]
    % Define the upper boundary (parabola)
    \addplot[name path=para, thick, domain=0:2.666, samples=100] {x^2/4};
    \addplot[name path=line, thick, domain=1:1.533, samples=100]{0.5*(x-1)};
    \addplot[name path= Sline, thick, dashed, domain=0.5:1.5, samples = 100]{x-0.75};
    
    % Define the piecewise lower boundary
    % Path goes from (0,0) to (4/3,0) to (8/3, 16/9)
    \path[name path=bot] (axis cs:0,0) -- (axis cs:1,0) -- (axis cs:1.533,0.2667) -- (axis cs:2.666, 1.777);
    
    % Fill the region R in grayscale
    \addplot[yellow!20] fill between[of=para and bot];
    
    % Draw the lower boundary line explicitly
    \draw[thick] (axis cs:0,0) -- (axis cs:1,0) -- (axis cs:1.533,0.2667) -- (axis cs:2.666,1.777);

    % Optional: Mark the tip of the region
    \filldraw (axis cs:2.666, 1.777) circle (1pt);

    \node[below, label={$S$}] at (0.6, 0.1) {};

    %Label regions W1 and W2
    %\node[below, label={$W_1$}] at (1.6, 0.005) {};
    %\node[below, label={$W_2$}] at (1.25,0.135) {};

\end{axis}
\end{tikzpicture}
\caption{The region $W_2$ with the line $b = a-3/4$ (dashed)}\label{fig:W2S}
\end{minipage}\hfill
\begin{minipage}{0.45\textwidth}
\centering
\begin{tikzpicture}
\begin{axis}[
    axis lines=middle,
    xlabel={$a$},
    ylabel={$b$},
    xmin=-0.1, xmax=1.1,
    ymin=-0.1, ymax=0.3,
    xtick={0, 0.5, 1.0},
    ytick={0, 0.1, 0.2},
    axis line style={->},
    axis on top,
    set layers,
    grid=none
]
    % 1. Define Upper Boundary (Parabola)
    \addplot[name path=para, thick, domain=0:1, samples=100] {x^2/4};
    
    % 2. Define Lower/Right Boundary
    % From (0,0) to (0.75,0) then to (1, 0.25)
    \path[name path=bot] (axis cs:0,0) -- (axis cs:0.75,0) -- (axis cs:1, 0.25);
    
    % 3. Shaded Region R
    \addplot[teal!20] fill between[of=para and bot];
    
    % 4. Explicit Boundary Lines (Thick Black)
    \draw[thick] (axis cs:0,0) -- (axis cs:0.75,0) -- (axis cs:1, 0.25);
    
    % 5. Region Label
    \node at (axis cs:0.5, 0.1) {$S$};

\filldraw (axis cs:1, 0.25) circle (1pt);
\node at (axis cs:0.85,0.25){$\left(1,\frac{1}{4}\right)$};
    
\end{axis}
\end{tikzpicture}
\caption{The region $S$}\label{fig:S}
\end{minipage}
\end{figure}

\begin{remark}\label{rem:nbh_of_S}
It is easy to see that if $(a,b)\in S$, then $a\in(0,1)$ and $b\in \left(0,\frac{1}{4}\right)$, see Figure~\ref{fig:S}.
\end{remark}

\section{\texorpdfstring{The derivative of $\tilde P_1$}{The derivative of P̃1}}\label{SEC:derivative_of_P1}

The goal of this section is to prove Proposition~\ref{prop:P}. As already noted at the beginning of Section~\ref{SEC:P1_0}, the fact that Proposition~\ref{prop:P} holds for metrics in $\mathcal B_H\cup\mathcal B_V$ follows from Propositions~\ref{prop:H} and \ref{prop:V}. Thus it remains to consider the metrics in $\mathcal B_P\setminus (\mathcal B_H\cup\mathcal B_V)$. Recall from Corollary~\ref{COR:BP_minus_BH} that 
$$
\mathcal B_P\setminus (\mathcal B_H\cup\mathcal B_V)\subset\{ g_{x,y,z,1}\; :\; (a,b)\in S\text{ and } x=x_4(a,b)   \},
$$
where $S$ is given in \eqref{EQ:set_S} and $x_4$ is given in \eqref{EQ:x3_and_x4}. We define the set in the right hand side as
$$
\hat S:=\{ g_{x,y,z,1}\; :\; (a,b)\in S\text{ and } x=x_4(a,b)   \}.
$$
First we note that $\tilde P_1$ is smooth at any point in $\hat S$. To show this, recall from Section~\ref{SEC:P1_0} that $b-a+x<0$ for any metric in $g_{x,y,z,1}\in\hat S$; in other words, the expression inside the absolute value of $P_1=\sqrt{H_1 V_1} +yz - 3\vert yz-y-z+x\vert$ is negative for all points in $\hat S$. Thus, the relevant formula for $\tilde P_1$ in this case is 
\begin{equation}\label{EQ:tilde_P_1}
\tilde P_1= \sqrt{\tilde H_1\tilde V_1} +4\frac{yz}{s^2} + 3\frac{x-y-z}{s}.
\end{equation}
Since $\tilde H_1$ and $\tilde V_1$ are smooth on $\R_+^4$, and moreover they are positive on $\hat S$ by Proposition~\ref{prop:possible_x4}, it follows that $\tilde P_1$ is smooth on $\hat S$. Then Proposition~\ref{prop:P} follows from the following result.

\begin{proposition}\label{prop:Pprime}  
    For all $g\in\hat S$ it holds that $\frac{d}{d\ell}\Big\vert_{\ell=0} \tilde P_1(g(\ell)) > 0$, where $g(\ell)$ denotes the Ricci flow of $g(0)=g$.
\end{proposition}

\begin{proof}
Let $g$ be an arbitrary metric in $\hat S$ and let $g(\ell)$ be its Ricci flow, i.e. $g(0)=g$. As we have done in the proofs of Propositions~\ref{prop:Hprime} and \ref{prop:Vprime}, to compute the derivative of $\tilde P_1(g(\ell))$ at $\ell=0$  we use the chain rule:
\begin{equation}\label{EQ:chain_for_P1}
        \frac{d}{d\ell}\Big\vert_{\ell=0}  \tilde P_1(g(\ell)) =\nabla  \tilde P_1(g(0)) g'(0). 
\end{equation}
From \eqref{EQ:RF_4param} we know that the Ricci flow $g(\ell)$ is the family $g_{x(\ell),y(\ell),z(\ell),s(\ell)}$ determined by $(x',y',z',s')=(-2r_x x,-2r_y y,-2r_z z, -2r_s s)$. At $\ell=0$, we have $(x,y,z,s)=(x_4,y,z,1)$, so $g'(0)=(-2r_x x_4,-2r_y y,-2r_z z,-2r_s)^T$.

The gradient $\nabla\tilde P_1=\left(\frac{\partial\tilde P_1}{\partial x}  , \frac{\partial\tilde P_1}{\partial y} , \frac{\partial\tilde P_1}{\partial z} , \frac{\partial\tilde P_1}{\partial s}   \right)$ can be computed using the expression in \eqref{EQ:tilde_P_1}. When evaluating it at a metric of the form $g_{x,y,z,1}$ we find that:
\begin{align*}
    \frac{\partial\tilde P_1}{\partial x} &= {( H_1 V_1)^{-1/2}} \left(\frac{-6 x^{2}-2 \left(y-z\right)^{2}+9 x^{2} \left(x-\frac{y}{3}-\frac{z}{3}\right)}{ x^{2}}\right)+3,  \\
    \frac{\partial\tilde P_1}{\partial y} &=  4( H_1 V_1)^{-1/2} \left(\frac{ \left(1-\frac{3 x}{4}\right) \left(y+x-z\right)}{x }\right)+4 z-3,   \\
    \frac{\partial\tilde P_1}{\partial z} &= 4( H_1 V_1)^{-1/2} \left(\frac{ \left(1-\frac{3 x}{4}\right) \left(z+x-y\right)}{x  }\right)+4 y-3,   \\
    \frac{\partial\tilde P_1}{\partial s} &=  6( H_1 V_1)^{-1/2}\left(\frac{ \left(-\frac{3 x}{2}+1\right) \left(x^{2}-\frac{2 \left(y+z\right) x}{3}-\frac{\left(y-z\right)^{2}}{3}\right)}{ x}\right)-8 y z+3 y+3 z-3 x.
\end{align*}

Now, we compute $\nabla\tilde P_1(g(0)) (-2r_x x,-2r_y y,-2r_z z,-2r_s)^T$ and substitute $r_x,r_y,r_z,r_s$ by their values from \eqref{EQ:Ricci_eigen_4param} at the metric $g_{x,y,z,1}$. This gives an expression in the variables $x,y,z$ which is symmetric in $y,z$.

We make the change of variables from $(y,z)$ to $(a,b)$ given in \eqref{EQ:change_variables} and  obtain the following expression in the variables $x,a,b$:
\begin{align*}
\frac{d}{d\ell}\Big\vert_{\ell=0} \tilde P_1(g(\ell))=\frac{1}{x^2 b\sqrt{H_1V_1}} D_n(x, a,b),
\end{align*}
where
 $$
 D_n(x, a, b)=Fn+G,
 $$
with 
   \begin{align*}
     F:=&  16x b \big(\left(-3 a+6 b+3\right) x^{3}+\left(-3 a^{2}+6 a b+6 a-12 b-3\right) x^{2}\\
   &+\left(6 a^{3}-12 a^{2} b+8 a \,b^{2}-7 a^{2}+3 a b-4 b^{2}+2 a+6 b\right) x-2 a^{3}+a^{2}+8 a b-4b\big),
   \end{align*}
 and
\begin{align*} G:=& (-48ab + 144b^2 - 24b + 24)x^4\\ &+ (128b^3 - 48a^2 + 440ab - 752b^2 + 16a - 144b)x^3\\ &+ (48a^3b - 144a^2b^2 + 128ab^3 - 360a^2 b + 656ab^2 - 512b^3 + 208ab + 224b^2 - 64b)x^2\\ &+ (48a^4 - 56a^3 b - 16a^3 - 160a^2 b + 224ab^2 + 64ab - 128b^2)x\\ &- 24a^4 + 128a^2 b - 128b^2.
\end{align*}
Note that, in obtaining theses expressions, we used the fact that $\sqrt{H_1V_1}=-4b+3a-3x$.

Recall that our goal is to show that $ \frac{d}{d\ell}\mid_{\ell=0} \tilde P_1(g(\ell)) > 0$ for any $g(0)=g\in\hat S$. This obviously holds if and only if $D_n(x, a, b)>0$ for any $(a,b)\in S$ and for $x=x_4(a,b)$. Therefore, it suffices to prove that $F(x_4,a,b)n + G(x_4,a,b) > 0$ for all $(a,b)\in S$. This will follow from the following two facts:
\begin{itemize}
    \item $F(x_4,a,b)>0$ for all $(a,b)\in S$, which is the content of Proposition~\ref{PROP:F_positive_on_S},
    \item $-\frac{1}{2} F(x_4,a,b) + G(x_4,a,b) \geq 0$ for all $(a,b)\in S$, which is done in Proposition~\ref{PROP:-12F_G_nonnegative}.
\end{itemize}
These two results imply that $G\geq \frac{1}{2} F > 0$, so clearly $Fn+G > 0$ for all $n\geq 1$.
\end{proof}

\begin{proposition}\label{PROP:F_positive_on_S}
    For any $(a,b)\in S$ it holds that $F(x_4,a,b)>0$.
\end{proposition}

\begin{proof}
We will show that the zero set of $F(x_4,a,b)$ does not intersect the region $S$, so its sign is constant on $S$. We will also show that $F(x_4,a,b)$ is positive on the boundary curve $a^2-4b=0$ of $S$. It will follow that $F(x_4,a,b)$ is positive on $S$.

Let us consider the zero set of $F(x_4,a,b)$. In order to simplify the computations we consider instead the zero set of $\frac{F(x_4,a,b)}{16x_4b}$, which is the same since both $x_4 > 0$ and $b>0$ on $S$. To do so, we substitute $x=x_4$ in the expression of $F(x,a,b)$ and follow the strategy of Remark~\ref{rem:proc} for $\frac{F(x_4,a,b)}{16x_4b}=0$. After rearranging the expression, we find that the zero set of $\frac{F(x_4,a,b)}{16x_4 b}$ is a subset of the zero set of the following equation: 
$$
  108 b^{2} \left(a+1\right) \left(-b-1+a\right) \left(a-2 b-1\right)^{2}\left(a^{2}-4 b\right) \left(4 a^{2}+8 a b-16 b^{2}-4 a-16 b-1\right) =0.
$$
For all the factors except the last two, it is clear that their zero set does not intersect $S$.  Thus, we need only analyze the last two factors.

We begin with the zero set of the factor $(a^2-4b)$. We will show that, in fact, if $a^2 - 4b = 0$ and $(a,b)\in S$, then $F(x_4,a,b)>0$. As we have computed in the proof of Proposition~\ref{prop:possible_x4}, along the boundary curve $b=\frac{a^2}{4}$ we have $x_4 = -\frac{a(a^3-6a^2+9a-8)}{6(a^2-2a+2)}$.

Substituting this into $F(x_4,a,b)$ and then substituting $b = \frac{a^2}{4}$ into the resulting expression, we obtain 
\begin{equation}\label{eqn:boundarypos} 
\frac{a^{6} \left(a^{3}-6 a^{2}+9 a-8\right)^{2} \left(a+1\right) \left(a^{3}-3 a^{2}+3 a+1\right) \left(a-2\right)^{2}}{216 \left(a^{2}-2 a+2\right)^{3}}.
\end{equation}
The factor $(a^{2}-2 a+2)$ in the denominator has negative discriminant, so it is positive for all $a$.  Therefore, the sign of \eqref{eqn:boundarypos} is determined by the sign of the numerator. Since $a\in (0,1)$ on $S$, the factors $a^6$, $(a+1)$ and $(a-2)^2$ are positive. Moreover,  from the proof of Lemma \ref{LEM:SimplifyGamma},  we know that $(a^3-6a^2 + 9a-8)< 0$ on $(0,1)$, so $(a^3-6a^2 + 9a-8)^2 > 0$.  Finally, the factor $(a^3 -3a^2 + 3a + 1)$ has the value $1> 0$ at $a=0$ and derivative $3(a-1)^2 \geq 0$, so it is also positive on $(0,1)$.  Thus, \eqref{eqn:boundarypos} is positive on $S$, as claimed.

To complete the proof that $F>0$ on $S$, it remains to investigate the last factor $(4a^2 + 8ab - 16b^2 -4a-16b-1)$. We shall show that its zero set does not intersect $S$. Note that the graph of $4 a^{2}+8 a b-16 b^{2}-4 a-16 b-1=0$ is a hyperbola.  In particular, no component of its graph is contained in any compact set.  It follows that if the graph intersects $S$, it must intersect the boundary.  We now show that, in fact, it does not intersect the boundary. Observe that $\partial S$ consists of three arcs with equations  $b=\frac{a^2}{4}$ with $a\in [0,1]$, $b=0$ with $a\in \left[0,\frac{3}{4}\right]$, and $b=a-\frac{3}{4}$ with $a\in\left[\frac{3}{4},1\right]$.

First, substituting $b= \frac{a^2}{4}$ into $4 a^{2}+8 a b-16 b^{2}-4 a-16 b-1 = 0$, we obtain $-(a+1)(a^3-3a^2+3a+1)=0$.  This has no solutions on $[0,1]$, so there the graph does not intersect the boundary of $S$ along this component.

Second, substituting $b= 0$ into $4 a^{2}+8 a b-16 b^{2}-4 a-16 b-1 = 0$, we find $4a^2 - 4a - 1 = 0$, which has no solutions on $\left[0,\frac{3}{4}\right]$.

Finally, substituting $b = a-\frac{3}{4}$ into $4 a^{2}+8 a b-16 b^{2}-4 a-16 b-1 = 0$, we obtain $-2(a+1)(2a-1)= 0$.  This has no solutions in $\left[\frac{3}{4},1\right]$.
\end{proof}

To complete the proof of Proposition~\ref{prop:Pprime} it remains to show the following result. 

\begin{proposition}\label{PROP:-12F_G_nonnegative}  
For any $(a,b)\in S$ it holds that $-\frac{1}{2} F(x_4,a,b) + G(x_4,a,b) \geq 0$.
\end{proposition}

The proof will make use of some auxiliary results. As usual, we shall determine the zero set $-\frac{1}{2} F(x_4,a,b) + G(x_4,a,b) = 0$. Substituting $x = x_4$ into $-\frac{1}{2}F + G$ and applying Remark~\ref{rem:proc} to $-\frac{1}{2}F + G=0$, we obtain an equation
$$
H(a,b)=0
$$
where
\begin{align*}
H(a, b)=1728 \, b^{2} (-b - 1 + a) (a^{2} - 4b)(a - 2b - 1)^{4}j(a, b), 
\end{align*}
with 
\begin{align*}
j(a, b)=&540 a^{8} b - 4608 a^{7} b^{2} + 15984 a^{6} b^{3} - 28800 a^{5} b^{4} + 28416 a^{4} b^{5} - 14592 a^{3} b^{6}\\ 
& + 3072 a^{2} b^{7} - 432 a^{8}+ 252 a^{7} b + 10008 a^{6} b^{2} - 33504 a^{5} b^{3} + 38304 a^{4} b^{4}\\
&- 5376 a^{3} b^{5} - 22528 a^{2} b^{6} + 17408 a b^{7} - 4096 b^{8} + 2520 a^{7} - 7083 a^{6} b\\
&- 5424 a^{5} b^{2} + 24036 a^{4} b^{3} + 8016 a^{3} b^{4} - 57088 a^{2} b^{5} + 47104 a b^{6} - 12288 b^{7} \\
&- 5676 a^{6} + 10251 a^{5} b + 18612 a^{4} b^{2}
- 29984 a^{3} b^{3} - 27888 a^{2} b^{4} + 50240 a b^{5}\\
&- 16384 b^{6} + 6368 a^{5} + 376 a^{4} b - 35842 a^{3} b^{2} + 7504 a^{2} b^{3} + 48128 a b^{4} \\
& - 25728 b^{5} - 3920 a^{4} - 7456 a^{3} b + 14692 a^{2} b^{2} + 29040 a b^{3} - 29632 b^{4}\\
&+ 1408 a^{3} + 3840 a^{2} b + 2272 a b^{2} - 12768 b^{3} - 256 a^{2} - 768 a b - 576 b^{2}. 
\end{align*}
To prove Proposition \ref{PROP:-12F_G_nonnegative}, we need the following result. To state it, let $Z_H$ and $Z_j$ denote the zero sets of $H$ and $j$, respectively. 

\begin{proposition}\label{prop:3components-new-version}  
The set $S\setminus Z_H$ has precisely two connected components.  The points $\left(\frac{19}{25},\frac{1}{100}\right)$ and $\left(\frac{9}{10}, \frac{3}{20}\right)$ lie on the closures of the two different components.
\end{proposition}

Assuming this proposition, the proof of Proposition \ref{PROP:-12F_G_nonnegative} is straightforward. 

\begin{proof}[Proof of Proposition \ref{PROP:-12F_G_nonnegative} assuming Proposition \ref{prop:3components-new-version}]
     The zero set of $-\frac{1}{2} F+ G$ is a subset of the zero set of $H$, so $-\frac{1}{2}F + G$ has a constant sign on each component of $S\setminus Z_H$. We substitute each of the two points of Proposition \ref{prop:3components-new-version} into $-\frac{1}{2} F(x_4,a,b) + G(x_4,a,b)$, and we obtain the positive values $\frac{9294208}{6591796875}$ and $\frac{87381}{250000}$, respectively.  By continuity, it follows that $-\frac{1}{2}F + G$ is positive on $S\setminus Z_H$, and hence non-negative on $S$.
\end{proof}

We now work towards proving Proposition~\ref{prop:3components-new-version}. The set $Z_H$ is the union of the zero sets of each of the factors.  The zero sets of the factors $b^2, (-b-1+a)$, $(a^2 - 4b)^2$ and $(a-2b-1)^4$ are disjoint from $S$ or only intersect along the boundary, so removing them does not change the number of components. Thus we need to study the intersection of $Z_j$ and $S$, whose analysis is significantly more difficult. This will be done rigorously in a series of lemmas, namely Lemmas~\ref{LEM:onesingular}, \ref{LEM:isolated}, \ref{LEM:nocircles} and \ref{LEM:Zj-Intersect-Bound}, after which we prove Proposition~\ref{prop:3components-new-version}. To provide intuition, we include a numerically estimated plot of $Z_j$ and $S$ in Figure \ref{fig:Zj}.

To state the next result, note that $Z_j$ is a real algebraic curve and recall that a point is called \textit{singular} if the partial derivatives $j_a$ and $j_b$ both vanish. By the implicit function theorem, a non-singular point is a (smooth) manifold point of $Z_j$.  Thus, in light of Remark \ref{rem:nbh_of_S}, the following result asserts that $Z_j$ is a smooth manifold in an open neighborhood of $S$, except at $(0,0)$.

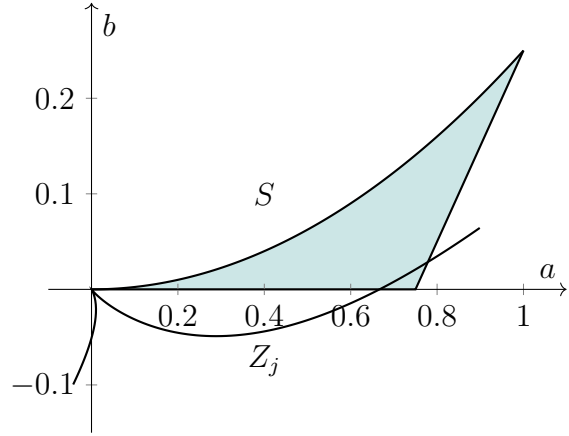
\begin{figure}
\begin{center}

\begin{tikzpicture}
\begin{axis}[
    axis lines=middle,
    xlabel={$a$},
    ylabel={$b$},
    xmin=-0.1, xmax=1.1,
    ymin=-0.15, ymax=0.3,
    xtick={0, 0.2, 0.4, 0.6, 0.8, 1.0},
    ytick={-0.1, 0, 0.1, 0.2},
    axis line style={->},
    axis on top,
    grid=none
]
    % 1. Upper Boundary: Parabola
    \addplot[name path=para, thick, domain=0:1, samples=100] {x^2/4};
    
    % 2. Bottom Boundary of R
    \path[name path=bot] (axis cs:0,0) -- (axis cs:0.75,0) -- (axis cs:1, 0.25);
    
    % 3. Shaded Region R
    \addplot[teal!20] fill between[of=para and bot];
    
    % 4. Explicit Boundary Lines of R
    \draw[thick] (axis cs:0,0) -- (axis cs:0.75,0) -- (axis cs:1, 0.25);
    
    % 5. Reference the CSV directly
    % [skip coords between index={start}{end}] can be used to sample data
    % or you can use 'restrict x to domain' to filter values
    \addplot[
        thick, 
        color=black,
        restrict x to domain=-0.1:0.9, % Filters data within TikZ
        each nth point=100           % Reduces memory usage by plotting every 10th point
    ] table [x=a, y=b, col sep=comma] {curve_data.csv};

    % Region Label
    \node at (axis cs:0.4, 0.1) {$S$};
    \node at (axis cs:0.4, -0.075) {$Z_j$};

\end{axis}
\end{tikzpicture}
\caption{The region $S$ with $j=0$}\label{fig:Zj}
\end{center}
\end{figure}

\begin{lemma}\label{LEM:onesingular} 
There is a number $\delta>0$ for which $(0,0)$ is the only singular point of $Z_j$ with $(a,b)\in \left(-\delta \vphantom{\frac{1}{4}}, 1+\delta\right)\times  \left(-\delta, \frac{1}{4} + \delta\right)$. 
\end{lemma}

\begin{proof}
The singular points are exactly the common solutions to $j = 0$, $j_a = 0$ and $j_b = 0$. Since the lowest degree term of $j$ has total degree $2$, it follows that $(0,0)$ is a singular point.

We now show that there are no other singular points in the appropriate region. To do so, let $(a_0,b_0)\in \left[0\vphantom{\frac{1}{4}},1\right]\times \left[0,\frac{1}{4}\right]$ be a common zero of $j$ and $j_a$. From Section~\ref{SEC:resultant_Sturm} we know that the resultant $\operatorname{res}_a(j,j_a)\in\mathbb Z [b]$ must vanish at $b_0$. Applying Sturm's Theorem~\ref{T:Sturm_root} (see also Remark~\ref{REM:Sturm}) to $\operatorname{res}_a(j,j_a)$ on the slightly larger interval $\left(-\frac{1}{100}, \frac{1}{4}\right]$, we find that there is exactly one value of $b$ in this interval for which $\operatorname{res}_a(j,j_a)$ vanishes. Since we know that $b=0$ is such a zero (with $a=0$), it follows that $b_0$ must be equal to $0$.  By continuity, it follows that there is a number $\delta_b \in \left(0,\frac{1}{100}\right)$ with the property that, on the interval $\left[-\delta_b, \frac{1}{4} + \delta_b\right]$, the only common zero $(a_0,b_0)$ of $j$ and $j_a$ has $b_0=0$, so it is of the form $(a_0,0)$.

Next, suppose that $(a_0,0)$ with $a_0\in [0,1]$ satisfies $j_b=0$. Applying Sturm's Theorem to $j_b(a,0)$, which is a single-variable polynomial in $a$, shows that it has exactly one real root on the larger interval $\left(-\frac{1}{100},1\right]$. Since $a=0$ is one such zero, it must be the only one, so $a_0=0$.  Again, by continuity, there is a number $\delta_a\in \left(0,\frac{1}{100}\right)$ with the property that on the interval $[-\delta_a, 1 + \delta_a]$  the only zero $(a_0,0)$ of $j_b$ is $(0,0)$. In particular, the only common zero $(a_0,0)$ to $j$, $j_a$ and $j_b$ is $(0,0)$.

Taking $\delta = \min\{\delta_a,\delta_b\}$ completes the proof.
\end{proof}

We now analyze $Z_j$ in a  neighborhood of $(0,0)$.

\begin{lemma}\label{LEM:isolated} 
There is a neighborhood $U$ of $(0,0)$ for which the only point of $Z_j\cap \overline{S}\cap U$  is $(0,0)$.
\end{lemma}

\begin{proof}  
We will show that there exists a neighborhood $U$ of $(0,0)$ such that $j$ has no zeros with both coordinates non-negative, other than $(0,0)$. In other words, the only point of $Z_j\cap\R^2_{\geq 0}\cap U$ is $(0,0)$, where $\R^2_{\geq 0}$ consists of the points $(a,b)$ with $a,b\geq 0$. This implies the claim in the statement because of the obvious inclusion $\overline{S}\subset\R^2_{\geq 0}$.

In this proof we use the following notation. Given non-negative integers $\alpha,\beta$, let $c_{\alpha,\beta}$ denote the coefficient of $a^\alpha b^\beta$ in the polynomial $j$. Let $j_2=\sum_{\alpha + \beta=2} c_{\alpha,\beta} a^\alpha b^\beta$ be the degree $2$ part of $j$. Further, let $M= \sum_{\alpha+\beta>2} |c_{\alpha,\beta}|$.

The behavior of $j$ near $(0,0)$ is roughly controlled by its lowest degree term $j_2$, which equals
$$
j_2(a,b)=-256 a^{2}-768 a b-576 b^{2}=-64(2a+3b)^2.
$$
This polynomial vanishes along the graph of $2a+3b=0$, which only intersects $\overline{S}$ at $(0,0)$. Moreover, we have that
\begin{equation}\label{EQ:bound_on_j2}
|j_2(a,b)|>64(2a+2b)^2=256(a+b)^2>256(a^2+b^2)\quad\text{for all }(a,b)\in\R_{\geq 0}^2.
\end{equation}
Now observe that $|j|=|-j_2 -j +j_2| \geq |j_2| - |j-j_2|$ by the triangle inequality, so we need to estimate $|j-j_2|$. We first claim that $|j(a,b)-j_2(a,b)| \leq M\sqrt{a^2+b^2}^3$ for all $(a,b)$ with $\sqrt{a^2 +b^2}\leq 1$. Indeed, we have:
\begin{align*}
|j(a,b)-j_2(a,b)| &\leq \sum_{\alpha+\beta > 2} \left|c_{\alpha+\beta} a^\alpha b^\beta\right|\\
&\leq \sum_{\alpha+\beta>2} \left|c_{\alpha,\beta}\right| \sqrt{a^2+b^2}^\alpha \sqrt{a^2+b^2}^\beta\\ &=  \sum_{\alpha+\beta>2} |c_{\alpha,\beta}| \sqrt{a^2+b^2}^{\alpha+\beta}\\
&\leq \sum_{\alpha+\beta>2} |c_{\alpha,\beta}| \sqrt{a^2+b^2}^3\\ &=  M \sqrt{a^2+b^2}^3.
\end{align*}  
Here, the first line follows from the triangle inequality, the second follows because both $a,b\leq \sqrt{a^2+b^2}$, and the fourth follows from the fact that $\alpha+\beta \geq 3$ and $\sqrt{a^2+b^2}\leq 1$.

Now, take any positive number $r$ with $r < \frac{256}{M}$ and let $U$ be the open ball around $(0,0)$ with radius $r$. Note that $M>256$, so $r<1$ and thus $U$ is contained in the unit disk $\sqrt{a^2 +b^2}\leq 1$. In particular $|j-j_2|\leq M\sqrt{a^2+b^2}^3$ on $U$. Using that $r < \frac{256}{M}$, we conclude that
\begin{equation}\label{EQ:bound_on_j_minus_j2}
j(a,b)-j_2(a,b)< 256(a^2+b^2 ),\quad\text{for all }(a,b)\in U\setminus \{(0,0)\}.
\end{equation}
Now we use \eqref{EQ:bound_on_j2} and \eqref{EQ:bound_on_j_minus_j2} in the following chain of inequalities on $\R_{\geq 0}^2\cap (U\setminus\{(0,0)\})$:
\begin{align*} 
|j(a,b)| &\geq |j_2(a,b)| - |j(a,b)-j_2(a,b)|\\
&> 256(a^2+b^2) - 256(a^2+b^2)\\
&= 0,
\end{align*} 
so $j(a,b)\neq 0$ on $\R_{\geq 0}^2\cap (U\setminus\{(0,0)\})$ as claimed.
\end{proof}

We note that $\overline{S}\subseteq \mathbb{R}^2$ is an  embedded smooth manifold with boundary except at the points $(0,0), \left(\frac{3}{4},0\right)$, and $\left(1,\frac{1}{4}\right)$. While $\left(\frac{3}{4},0\right),\left(1,\frac{1}{4}\right)\notin Z_j$, we do know that $(0,0)\in Z_j$. Fortunately, from Lemma~\ref{LEM:isolated} we know that $Z_j\cap\overline{S}$ does not intersect a neighborhood of $(0,0)$. This implies that any point in $Z_j\cap\overline{S}$, apart from $(0,0)$, occurs at smooth (possibly boundary) points of $\overline{S}$. Hence, the implicit function theorem implies that, apart from $(0,0)$, the set $Z_j\cap\overline{S}$ is a union of smoothly embedded manifolds (possibly with boundary) of $\overline S$. As such, $Z_j\cap\overline{S}$ is a disjoint union of the point $(0,0)$, of $0$-manifolds (i.e. isolated points) and smoothly embedded $1$-manifolds. Let us discuss them.

As for the isolated points in $Z_j\cap\overline{S}$ other than $(0,0)$, Lemma~\ref{LEM:onesingular} implies that any such point would lie on the boundary $\partial S$ (and would correspond to a tangential intersections of $\overline S$ and $Z_j$). As for the embedded $1$-manifolds of $Z_j$, they must be either smooth circles or intervals, possibly including one or both endpoints. Since $S$ is compact and zero sets are closed, the intersection of $Z_j$ with $S$ must be compact, hence we can only get smooth circles or closed intervals. Moreover, again by Lemma~\ref{LEM:onesingular}, in the case of embedded intervals the endpoints must lie in the boundary of $\overline S$. The next proposition indicates that there are actually no circle components.

\begin{lemma}\label{LEM:nocircles} 
The set $Z_j\cap \overline{S}$ contains no embedded closed curves.
\end{lemma}

\begin{proof}  
Suppose for a contradiction that such a closed curve $C$ exists.  By compactness of $C$, the map $C\rightarrow \mathbb{R}$ taking a point in $C\subseteq \overline{S}\subseteq \mathbb{R}^2$ to its $b$-coordinate achieves an absolute maximum at some $(a_0,b_0)\in C$ with $b_0 > 0$.  Being an absolute maximum, the tangent line to $C$ at $(a_0,b_0)$ must be horizontal, so the slope $-\frac{j_a}{j_b}$ vanishes, so that $j_a (a_0,b_0)= 0$.  But we showed in the proof of Lemma~\ref{LEM:onesingular} that any common solution $(a,b)$ of the system of equations $j = 0$ and $j_a = 0$ must have $b=0$.  This contradicts the fact that $b_0 > 0$.
\end{proof}

Because there are no circular components of $Z_j\cap \overline{S}$, we need only determine the isolated points of $Z_j\cap \overline{S}$ and the closed interval components. In Lemma~\ref{LEM:Zj-Intersect-Bound} we shall show that  $Z_j\cap \partial \overline{S}$ consists of precisely three points. In particular, together with Lemmas~\ref{LEM:onesingular} and \ref{LEM:isolated} it will follow that $Z_j\cap \overline{S}$ equals either these three points or the point $(0,0)$ together with one line segment joining the two other points. We will later see in the proof of Proposition~\ref{prop:3components-new-version} that, in fact, the second possibility holds.

\begin{lemma}\label{LEM:Zj-Intersect-Bound}
The set $Z_j\cap \partial \overline{S}$ consists of precisely three points: $(0,0)$, one of the form $(a,0)$ for some $a\in \left(0,\frac{3}{4}\right)$ and one of the form $\left(a,a-\frac{3}{4}\right)$ for some $a\in \left(\frac{3}{4},1\right]$. 
\end{lemma}

\begin{proof} 
The set $\overline{S}$ contains three boundary arcs with equations  $b=\frac{a^2}{4}$ with $a\in [0,1]$, $b=0$ with $a\in \left[0,\frac{3}{4}\right]$, and $b=a-\frac{3}{4}$ with $a\in\left[\frac{3}{4},1\right]$.

To begin with, if we substitute $b= \frac{a^2}{4}$ into $j = 0$, Sturm's Theorem~\ref{T:Sturm_root} (see also Remark~\ref{REM:Sturm}) shows that the resulting function of $a$ has no zeros on $(0,1]$.

If we substitute $b= 0$ into $j$, Sturm's Theorem shows the resulting function of $a$ has precisely one zero on $\left(0,\frac{3}{4}\right]$  Moreover, this zero must lie in the interval $\left(0,\frac{3}{4}\right)$ because, as it is easy to verify, the point $\left(\frac{3}{4},0\right)$ does not belong to $Z_j$.

Finally, if we substitute $b=a-\frac{3}{4}$ into $j$, Sturm's Theorem shows the resulting function of $a$ has precisely one zero on $\left(\frac{3}{4},1\right]$.
\end{proof}

We are finally ready to prove Proposition \ref{prop:3components-new-version}.

\begin{proof}[Proof of Proposition \ref{prop:3components-new-version}]
Consider the points $\left(\frac{19}{25}, \frac{1}{100}\right)$ and $\left(\frac{9}{10}, \frac{3}{20}\right)$.  These points lie on the curve $b = a-\frac{3}{4}$ which lies in the boundary of $S$.  If we substitute these points into $j$, we find that the outputs are respectively $\frac{9274567512848}{3814697265625} > 0$ and $-\frac{1820807829}{125000000}  <0$.  It follows that they lie in different components of $\overline{S}\setminus Z_j$. By continuity, it follows that $j$ changes sign on $S$. Thus, there are at least two components of $\overline{S}\setminus Z_j$. We complete the proof by arguing that there are only two components.

We showed above (see Lemma~\ref{LEM:Zj-Intersect-Bound} and the preceding discussion) that $Z_j\cap\overline{S}$ equals either three isolated points lying in $\partial S$ or the $(0,0)$ together with one line segment joining the two other points. In the first possibility it is clear that $\overline S\setminus Z_j$ is connected, which is a contradiction with the fact that $\overline S\setminus Z_j$ has at least two components. Thus, $\overline S\setminus Z_j$ equals $(0,0)$ together with one line segment joining the two other points.  Now note that removing removing $(0,0)$ does not affect the number of components, so $\overline S\setminus Z_j$ has the same number of components as $\overline S$ with the interval with endpoints in $\partial S$ removed. It is easy to see that the interval divides $\overline S$ in two components, see for example \cite[Theorem~11.7]{Ne61}. It follows that $\overline S\setminus Z_j$ has two components, and the same holds for $S\setminus Z_j$.
\end{proof}

\section{Proof of Theorem~\ref{THM:homogeneous_spaces}}\label{sec:thmcd}

In this section we prove Theorem~\ref{THM:homogeneous_spaces}. We will use the following result.

\begin{lemma}\label{LEM:RF_coverings}
Suppose $\pi:(M_1,g_1)\rightarrow (M_2,g_2)$ is a Riemannian covering map of Riemannian manifolds $(M_i,g_i)$, $i\in \{1,2\}$.  Let $L_i\in (0,\infty]$ denote the maximal time of existence of the Ricci flow of the metric $g_i$.  Then $L_1 = L_2$.
\end{lemma}

\begin{proof}
Let $g_i(\ell)$ denote the Ricci flow of $g_i$.  Since $g_1= \pi^\ast g_2$ is invariant under the action of the  deck transformations group of $\pi$, so is $g_1(\ell)$ for all $\ell\in[0, L_1)$. Therefore, $g_1(\ell)$ descends to a curve of metrics $\overline{g}_1(\ell)$ on $M_2$ with $\overline{g}_1(0) = g_2$.  Since $\pi$ is a local isometry, $\overline{g}_1(\ell) $ satisfies the  Ricci flow equation on $M_2$ and by uniqueness of Ricci flow, we conclude that $\overline{g}_1(\ell) = g_2(\ell)$ for all $\ell\in [0,L_1)$.  In particular, this implies $L_1\leq L_2$.

Conversely, the metric $\pi^\ast g_2(\ell)$ satisfies the Ricci flow equation on $M_1$ with $\pi^\ast g_2(0) = g_1$, so $\pi^\ast g_2(\ell) = g_1(\ell)$ by uniqueness of the Ricci flow.  It now follows that $L_2\leq L_1$. 
\end{proof}

\begin{proof}[Proof of Theorem~\ref{THM:homogeneous_spaces}]  

First of all, in the simply connected case the proof has been mostly given in the introduction. Again, from the classification of positively curved homogeneous spaces, collected and completed in \cite{WZ18}, we know that such a space $M$ is diffeomorphic to a compact rank one symmetric space, a Berger space $B^7,B^{13}$, a Wallach space $W^6,W^{12},W^{24}$, or an Aloff-Wallach space $W_{p,q}^7$. Now, if $M$ has a homogeneous $\sec >0$ metric which evolves to planes with non-positive curvature, Theorem~\ref{THM:main} implies that $M$ cannot be a compact rank one symmetric space, and it cannot be $B^7$ either, since $B^7$ admits a unique homogeneous metric up to scaling and  this metric is Einstein \cite[p. 83]{Z-Survey}. Thus $M$ must be diffeomorphic to one of $W^6$, $W^{12}$, $W^{24}$, $W_{p,q}^7$ or $B^{13}$. Conversely, if $M$ is diffeomorphic to one of $W^6$, $W^{12}$, $W^{24}$, $W_{p,q}^7$ or $B^{13}$, it follows from \cite[Theorem~2]{CW15} and \cite[Theorems~A and B, and Subsection~3.5]{GZ25} that $M$ has a homogeneous $\sec >0$ metric which evolves to planes with non-positive curvature. 
   
Now let $M$ be a compact connected manifold admitting a homogeneous metric $g$ with $\sec>0$ which evolves to metrics with non-positively curved planes under the Ricci flow. In fact, the following proof works in the more general case where $g$ is locally homogeneous. Let $(\tilde M, \tilde g)$ be the universal cover of $(M, g)$. Then  $(\tilde M, \tilde g)$ is a homogeneous space \cite[Main~Theorem, p. 692]{S60}). Let $\tilde g(t)$ and $g(t)$ be the Ricci flows of $\tilde g$ and $g$, respectively. Lemma~\ref{LEM:RF_coverings} implies that $\tilde g(t)$ and $g(t)$ have the same maximal forward existence time. Consequently, the homogeneous metric $\tilde g$ on $\tilde M$ evolves to a metric with non-positively curved planes since the universal covering maps are local isometries. It follows from the simply connected case settled above that $\tilde M$ must be diffeomorphic to one of the spaces listed in the statement.

Now, we prove the converse. If $M$ is simply connected, the result has been proven above. Therefore, we assume that $M$ is a non-simply connected homogeneous manifold whose universal cover is one of $W^6$, $W^{12}$, $W^{24}$, $W_{p,q}^7$ or $B^{13}$. For the proof we rely on the classification of positively curved homogeneous spaces in \cite{WZ18}, and for simplicity we do not repeat this citation throughout the proof. In particular, $B^{13}$ does not admit a positively curved finite quotient, so it suffices to consider spaces $M$ whose universal cover $\tilde M$ is one of $W^6$, $W^{12}$, $W^{24}$ or $W_{p,q}^7$. We will show that the universal cover has a homogeneous metric $g_1$ of $\sec >0$ which, on the one hand evolves to metrics with non-positively curved planes, and on the other hand descends to a homogeneous metric on the quotient $M$ in the sense that $(\tilde M,g_1)\to (M,g_2)$ is a Riemannian covering. By the properties of coverings and of the Ricci flow it then follows that $g_2$ is a $\sec >0$ metric which evolves to metrics with non-positively curved planes, as desired. We treat each case separately. 

Suppose that $\tilde M$ is diffeomorphic to an Aloff--Wallach space $W^7_{p, q}$. First note that from \cite[Proposition 9, pg. 94]{On1},  if $M$ is diffeomorphic to the homogeneous presentation  $(\u{3}/\mathsf{T}^2_{p,q})/\Gamma$, for some finite subgroup $\Gamma\subseteq  \frac{{N}_{\u{3}}(\mathsf{T}^2_{p, q})}{\mathsf{T}^2_{p, q}}$, we can find a finite subgroup $\tilde \Gamma\subseteq \frac{N_{\su{3}}(\mathsf{S}^1_{p, q})}{\mathsf{S}^1_{p, q}}$ such that $M$ is diffeomorphic to $(\su{3}/\mathsf{S}^1_{p,q})/\tilde\Gamma$. Hence without loss of generality, we can assume that $M$ is given as a finite quotient of $\su{3}/\mathsf{S}^1_{p,q}$. 
 
 Consider the case $(p, q)=(1, 1)$. Let $g_1$ be a metric on $W^7_{1, 1}$ which is $\u{2}$-invariant and evolves under the Ricci flow to a metric with non-positively curved planes. The existence of such metrics is guaranteed by \cite[Theorem 2.8]{GZ25}. Since $M$ is diffeomorphic to $W^7_{1, 1}/\Gamma$, where $\Gamma$ is a finite subgroup of $\u{2}/\mathsf{S}^1_{1,1}$, the metric $g_1$ descends to a metric $g_2$ on $M$ with the desired properties. 

 Now, let $p\neq q$. Then there exists a $\mathsf{T}^2$-invariant metric $g_1$ on $W^7_{p, q}$ which evolves under the Ricci flow to a metric with non-positively curved planes \cite[Subsection~3.5]{GZ25}. Similarly, since $M$ is diffeomorphic to $W^7_{p,q}/\Gamma$, where $\Gamma$ is a finite subgroup of $\mathsf{T}^{2}/\mathsf{S}^1_{p,q}$, the metric $g_1$ descends to a metric $g_2$ on $M$ as in the statement of the theorem. 

Suppose now that $\tilde M$ is diffeomorphic to a Wallach space $W:=W^n$, for $n=6, 12, 24$. Since the arguments in all cases are analogous, we treat them simultaneously. Recall that, by Synge's Theorem, any non-trivial isometric quotient of an even-dimensional $\sec>0$ manifold must have fundamental group isomorphic to $\mathbb Z_2$. In the case of $W$, there are three $\mathbb{Z}_2$-quotients  which give rise to non-simply connected homogeneous spaces  with positive sectional curvature. These three $\mathbb{Z}_2$ subgroups are conjugate in $S_3=N_{G}(H)/H$, where $G/H$ denotes the corresponding homogeneous presentation of $W$. Hence the quotients of $W$ by these groups are isometric \cite[Lemma~2.5.6]{Wolf}. Therefore, we only need to deal with one of the quotients.

Let $g_1$ be a $G$-invariant right $\mathbb{Z}_2$-invariant  metric with $\sec>0$ which evolves under the Ricci flow to a metric with non-positively curved planes \cite[Theorem 2]{CW15}. Since $\mathbb{Z}_2\subseteq N_{G}(H)/H$, it follows that $g_1$ descends to a metric $g_2$ on $M$ as desired. 
\end{proof}

 We conclude this section with some remarks.
\begin{remark}
As explained in the proof of Theorem~\ref{THM:homogeneous_spaces}, the ``only if'' part works more generally for locally homogeneous spaces. More precisely, the proof shows that if $M$ admits a locally homogeneous metric $g$ with $\sec > 0$ whose Ricci flow evolves to metrics with non-positively curved planes, then the universal cover of $M$ is diffeomorphic to one of $W^6$, $W^{12}$, $W^{24}$, $W^7_{p,q}$ or $B^{13}$.    
\end{remark}

\begin{remark}
 B\"ohm and Wilking found homogeneous metrics on $W^{12}$ with $\sec>0$ which lose this property under the flow \cite{BW07}. Their strategy of proof differs from the one in \cite{CW15,GZ25}, and they in fact  show the stronger statement that the flow leaves the set of metrics of positive Ricci curvature. Since the metrics of B\"ohm and Wilking do not descend to the isometric $\mathbb Z_2$-quotients of $W^{12}$, it remains open whether the same  behavior occurs for these quotients.
\end{remark}

\end{document}